\documentclass[reqno, 11pt]{amsart}

\usepackage[text={155mm, 235mm},centering]{geometry}
\input diagxy
\input xy

\usepackage[active]{srcltx} 
\usepackage[mathscr]{eucal}
\usepackage[all]{xy}
\usepackage{mathrsfs}
\usepackage{xypic}
\usepackage{amsfonts,mathtools}
\usepackage{amsmath}
\usepackage{amsthm}
\usepackage{amssymb}
\usepackage{latexsym}
\usepackage{tabularx}
\usepackage{graphicx}
\usepackage{pict2e}
\usepackage[pagebackref]{hyperref}
\usepackage{tikz}
\usepackage[toc,page]{appendix}
\usepackage[scr=rsfs]{mathalfa}
\definecolor{ceruleanblue}{rgb}{0.16, 0.32, 0.75}
\hypersetup{colorlinks=true,allcolors=ceruleanblue}
\newtheorem{thm}{Theorem}[section]
\newtheorem{cor}[thm]{Corollary}
\newtheorem{lem}[thm]{Lemma}
\newtheorem{prop}[thm]{Proposition}

\theoremstyle{definition}
\newtheorem{defn}[thm]{Definition}
\theoremstyle{property}

\theoremstyle{remark}
\newtheorem{rem}[thm]{Remark}

\newtheorem{prob}[thm]{Problem}
\newtheorem{ex}[thm]{Example}

\DeclareMathOperator{\E}{\mathscr{E}}

\DeclareMathOperator{\K}{\mathcal{K}}

\DeclareMathOperator{\Id}{\mathrm{id}}

\allowdisplaybreaks

\numberwithin{equation}{section} 
\begin{document}

\title[]{Bott-Chern blow-up formula and bimeromorphic invariance of the $\partial\bar{\partial}$-Lemma for threefolds}

\author[S. Yang]{Song Yang}
\address{Center for Applied Mathematics, Tianjin University, Tianjin 300072, P. R.  China}%
\email{syangmath@tju.edu.cn}%

\author[X. Yang]{Xiangdong Yang}
\address{Department of Mathematics, Chongqing University, Chongqing 401331, P. R. China}
\email{math.yang@cqu.edu.cn}

\date{\today}


\begin{abstract}
The purpose of this paper is to study the bimeromorphic invariants of compact complex manifolds in terms of Bott-Chern cohomology.
We prove a blow-up formula for Bott-Chern cohomology.
As an application, we show that for compact complex threefolds the non-K\"{a}hlerness degrees,
introduced by Angella-Tomassini [Invent. Math. 192, (2013), 71-81], are bimeromorphic invariants.
Consequently, the $\partial\bar{\partial}$-Lemma on threefolds admits the bimeromorphic invariance.
\end{abstract}



\maketitle

\section{Introduction}

Let $X$ be a complex manifold then it admits a huge space of smooth Hermitian metrics.
In general, to study the geometry and topology of $X$ it is of usefulness to pick a special Hermitian metric on $X$,
such as the K\"{a}hler metric and the balanced metric.
Among these metrics, the most distinguished ones on complex manifolds are K\"{a}hler metrics.
In the past decades, there have been extensive study of K\"{a}hler geometry from differential and algebraic geometry points of view.
The existence of a K\"{a}hler metric on a compact complex manifold implies many special topological consequences,
and one of them is the $\partial\bar{\partial}$-Lemma.
According to the real homotopy theory, a direct result of the $\partial\bar{\partial}$-Lemma
is that the real homotopy type of a compact K\"{a}hler manifold is a formal consequence of the de Rham cohomology ring (cf. \cite{DGMS75}).
Following the notion of Popovici \cite{Pop15}, we say that a compact complex manifold $X$ is a \emph{$\partial\bar{\partial}$-manifold} if the $\partial\bar{\partial}$-Lemma holds on $X$.
It is noteworthy that the class of $\partial\bar{\partial}$-manifolds consists of Moishezon manifolds and compact complex manifolds in the class $\mathcal{C}$ of Fujiki, which are not K\"{a}hler in general (cf. \cite{Moi66,Fuj78}).
On one hand, a great number of examples of $\partial\bar{\partial}$-manifolds that are known so far carry the \emph{balanced metrics}.
On the other hand, there exist many examples of balanced manifolds on which the $\partial\bar{\partial}$-Lemma does not hold, for instance,
a classical example is the Iwasawa manifold (cf. \cite{Ang13}).

In complex geometry, there exist two important geometric operations: modification and deformation.
These operations provide us with useful methods to produce new complex manifolds from the original ones.
In 1960, Hironaka \cite{Hir60} constructed a compact non-K\"{a}hler threefold $X$, with a modification
$f:X\rightarrow\mathbb{C}\mathrm{P}^{3}$,
on which the $\partial\bar{\partial}$-Lemma holds.
This means that the K\"{a}hler metric on a compact complex manifold is not preserved by modifications.
On balanced metric, Alessandrini-Bassanelli \cite{AB93,AB95} proved that the balanced structure is stable under modifications.
Moreover, Popovici \cite{Pop13} showed that the strongly Gauduchon (sG) structure on a compact complex manifold admits the stability under modifications.
On the small deformations of K\"{a}hler manifolds,
a classical result of Kodaira-Spencer \cite{KS60} shows that small deformations of a compact K\"{a}hler manifold remain K\"{a}hler;
see also Rao-Wan-Zhao \cite{RWZ16,RWZ18} for a new proof.
The example of Iwasawa manifold shows that the balanced structure is not preserved under small deformations (cf. \cite{AB90});
whereas, it is stable under small deformations when the manifold satisfies some versions of the $\partial\bar{\partial}$-Lemma
(cf. \cite{FY11, RWZ16,Wu06} etc).
Particularly, the $\partial\bar{\partial}$-Lemma itself admits the stability under small deformations (cf. \cite{Voi02,Wu06,AT13}).

Assume that $X$ and $\tilde{X}$ are two compact complex manifolds with the same dimensions and let $f:\tilde{X}\rightarrow X$ be a modification.
According to a result of Deligne-Griffiths-Morgan-Sullivan \cite[Theorem 5.22]{DGMS75}, we know that if the $\partial\bar{\partial}$-Lemma holds for $\tilde{X}$ then so does for $X$ (see also \cite[Theorem 12.9]{Dem12}).
Conversely, it is natural to consider the following question (see also \cite{Ale17}):
if the $\partial\bar{\partial}$-Lemma holds for $X$, then whether it also holds for $\tilde{X}$?
In general, modifications are special cases of  bimeromorphic maps, we have the following
\begin{prob}\label{problem}
Is the $\partial\bar{\partial}$-Lemma a bimeromorphic invariant of compact complex manifolds?
Here the bimeromorphic invariance means that if $f:\tilde{X}\dashrightarrow X$ is a bimeromorphic map of compact complex manifolds then the $\partial\bar{\partial}$-Lemma holds on $\tilde{X}$ if and only if so does on $X$.
\end{prob}

In 2002, Abramovich-Karu-Matsuki-W{\l}odarczyk \cite[Theorem 0.3.1]{AKMW02} proved the Weak Factorization Theorem which asserts that any bimeromorphic map of compact complex manifolds $f:\tilde{X}\dashrightarrow X$ can be decomposed into a finite sequence of blow-ups and blow-downs with smooth centers.
Therefore, to prove the bimeromorphic invariance of the $\partial\bar{\partial}$-Lemma,
it is sufficient to show that the $\partial\bar{\partial}$-Lemma is stable under blow-ups.
From cohomogical point of view, Angella-Tomassini \cite{AT13} showed that the validity of the $\partial\bar{\partial}$-Lemma on a compact complex manifold can be characterized by Bott-Chern cohomology.
This implies that to prove that the blow-up of a $\partial\bar{\partial}$-manifold satisfies the $\partial\bar{\partial}$-Lemma
we need a blow-up formula for Bott-Chern cohomology.
Using a sheaf-theoretic approach, we prove the following result.

\begin{thm}\label{Thm1}
Let $X$ be a compact complex manifold with $\mathrm{dim}_{\mathbb{C}}\,X=n$.
Assume that $Z\subset X$ is a closed complex submanifold such that $\mathrm{codim}_{\mathbb{C}}\,Z=r\geq2$.
Then we have the following isomorphisms:
\begin{itemize}
  \item [(i)]
  $
  H_{BC}^{p,q}(\tilde{X})\cong
  H_{BC}^{p,q}(X)\bigoplus\biggl(H_{BC}^{p,q}(E)/(\pi|_{E})^{\ast}H_{BC}^{p,q}(Z)\biggr),
  $
  for $1\leq p,q \leq n$;
  \item [(ii)]
  $H^{p,q}_{BC}(\tilde{X})\cong H^{p,q}_{BC}(X)$, for $p=0$ or $q=0$.
\end{itemize}
Here $\pi:\tilde{X}\rightarrow X$ is the blow-up of $X$ along $Z$ and $E=\pi^{-1}(Z)$ is the exceptional divisor.
\end{thm}

Note that on compact complex surfaces the $\partial\bar{\partial}$-Lemma, the balanced condition, and the K\"{a}hler condition are equivalent;
however, if the complex dimension is larger than 2 those conditions are not equivalent sharply.
In particular, there exist examples of compact balanced manifolds on which the $\partial\bar{\partial}$-Lemma does not hold.
Besides, all $\partial\bar{\partial}$-manifolds known so far are balanced manifolds,
and hence this leads to a conjecture that every $\partial\bar{\partial}$-manifold carries a balanced metric (cf. \cite{Pop15, RZ16}).
In 2012, Fu-Li-Yau \cite[Theorem 1.2]{FLY12} proved that each Clemens manifold has a balanced metric;
moreover, recently Friedman \cite[Theorem 3.10]{Fri17} showed that for a general Clemens manifold the $\partial\bar{\partial}$-Lemma holds.
In their paper \cite{AT13}, Angella-Tomassini introduced a series of numerical invariants called \emph{non-K\"{a}hlerness degrees} for compact complex manifolds.
As an application of Theorem \ref{Thm1}, we consider the bimeromorphic invariance of non-K\"{a}hlerness degrees for compact complex threefolds
and we give a definite answer to Problem \ref{problem} in this case.

\begin{thm}\label{Thm2}
Let $f:\tilde{X}\dashrightarrow X$ be a bimeromorphic map of compact complex threefolds.
Then for any $k\in\{0,1,2,3,4,5,6\}$ the $k$-th non-K\"{a}hlerness degree of $\tilde{X}$ is equal to the $k$-th non-K\"{a}hlerness degree of $X$.
Moreover, $\tilde{X}$ is a $\partial\bar{\partial}$-manifold if and only if $X$ is a $\partial\bar{\partial}$-manifold.
\end{thm}

It is well-known that $(p,0)$- and $(0,q)$-Hodge numbers are bimeromorphic invariants of compact complex manifolds (see \cite[Corollary 1.4]{RYY17} for a uniform proof).
Due to the symmetric property of Bott-Chern cohomology the $(0,q)$-Bott-Chern number equals to the $(q,0)$-Bott-Chern number.
Analogously, on the $(p,0)$-Bott-Chern number for a compact complex manifold $X$,
a natural problem is:

\begin{prob}
Does the $(p,0)$-Bott-Chern number admit the bimeromorphic invariance?
\end{prob}

As an application of the second part of Theorem \ref{Thm1}, we show that the $(p,0)$-Bott-Chern numbers of a compact complex manifold are bimeromorphic invariants (Corollary \ref{p-0-inv}).

An outline of this paper is organized as follows.
We devote Section \ref{Prelim} to the preliminaries of Bott-Chern cohomology and the cohomological characterization of the $\partial\bar{\partial}$-Lemma
on a compact complex manifold.
In Section \ref{blow-up-formula-BC}, we prove a blow-up formula for Bott-Chern cohomology.
In Section \ref{invar-ddbar}, we show that the non-K\"{a}hlerness degrees of threefolds are bimeromorphic invariants.
Finally, in Section \ref{remark} we give a brief remark on the bimeromorphic invariance of
$\partial\bar{\partial}$-Lemma on compact complex manifolds in general.

Recently, we were informed that, using different techniques, D. Angella, T. Suwa, N. Tardini and A. Tomassini also proved the stability of the $\partial\bar{\partial}$-Lemma under blow-ups with centers satisfying the $\partial\bar{\partial}$-Lemma and having holomorphically contractible neighbourhoods \cite[Theorem 2.1]{ASTT17}.
The approach in \cite{ASTT17} is a \v{C}ech cohomology theory for Dolbeault cohomology which gives rise to a more topological method from a complimentary point of view.

\subsection*{Acknowledgement}
We would like to express our gratitude to Prof. Guosong Zhao and Prof. Xiaojun Chen for their constant encouragements and supports,
and sincerely thank the Departments of Mathematics, Universit\`{a} degli Studi di Milano and Department of Mathematics, Cornell University for the hospitalities during our respective visits.
Specially, we would like to thank Sheng Rao for many useful discussions and comments which the present version is benefitted from.
Last, but not least, we sincerely thank Prof. D. Angella, Dr. L. Meng  and Dr. J. Stelzig for sending us their paper \cite{ASTT17}, \cite{Men18} and \cite{Ste18}, respectively.
This work is partially supported by the NSFC (Grant No. 11571242, No. 11701414 and No. 11701051) and the China Scholarship Council.

\section{Preliminaries}\label{Prelim}
\subsection{Cohomologies of complex manifolds}
Assume that $(X,J)$ is a compact complex manifold of complex dimension $n$.
Let $\mathcal{A}^{k}(X;\mathbb{C})$ denote the space of $\mathbb{C}$-valued differential forms of degree $k$ on $X$.
The complex structure $J$ determines a decomposition of the exterior differential $d=\partial+\bar{\partial}$;
moreover, it induces a natural decomposition of the complexified cotangent bundle
\[
T^{\ast}_{\mathbb{C}}X=T^{*\prime}X\oplus T^{*\prime\prime}X.
\]
The decomposition above gives a decomposition on the space of differential forms
$$
\mathcal{A}^{k}(X;\mathbb{C})=\bigoplus_{p+q=k}\mathcal{A}^{p,q}(X;\mathbb{C}),
$$
where $\mathcal{A}^{p,q}(X;\mathbb{C})$ is the space of smooth $(p,q)$-forms on $X$.
Furthermore, we get a natural bi-differential $\mathbb{Z}_{2}$-graded algebra, namely, the double complex of $X$ denoted by
$
(\mathcal{A}^{\bullet,\bullet}(X;\mathbb{C});\partial,\bar{\partial}).
$
In particular, there are several important cohomologies $X$ associated to
$
(\mathcal{A}^{\bullet,\bullet}(X;\mathbb{C});\partial,\bar{\partial})
$
as follows (cf. \cite{Dol53, BC65, Aep65}).
\begin{itemize}
  \item The \emph{de Rham cohomology}
  $$
  H^{\bullet}_{dR}(X;\mathbb{C}):=\frac{\ker\,d}{\mathrm{im}\,d};
  $$
  \item The \emph{$\partial$-cohomology} and \emph{Dolbeault cohomology}
  $$
  H^{\bullet,\bullet}_{\partial}(X):=\frac{\ker\,\partial}{\mathrm{im}\,\partial},
 \quad
  H^{\bullet,\bullet}_{\bar{\partial}}(X):=\frac{\ker\bar{\partial}}{\mathrm{im}\,\bar{\partial}};
  $$

  \item The \emph{Bott-Chern cohomology} and \emph{Aeppli cohomology}
  $$
  H^{\bullet,\bullet}_{BC}(X):=\frac{\ker\partial\cap\ker\bar{\partial}}{\mathrm{im}\,\partial\bar{\partial}},
  \quad
  H^{\bullet,\bullet}_{A}(X):=\frac{\ker\partial\bar{\partial}}{\mathrm{im}\,\partial+\mathrm{im}\,\bar{\partial}}.
  $$
\end{itemize}
From definition, the Bott-Chern and Aeppi cohomology admit the symmetric property by complex conjugation.

It is noteworthy that all the above cohomologies can be thought of as \emph{bi-graded algebras over $\mathbb{C}$} with respect to the wedge product.
In particular, the identity map induces the following natural maps of bi-graded $\mathbb{C}$-vector spaces
\begin{equation}\label{complex-cohom-diagram}
\xymatrix@C=0.5cm{
 & H^{\bullet,\bullet}_{BC}(X) \ar[d]\ar[ld]\ar[rd] & \\
 H^{\bullet,\bullet}_{\partial}(X) \ar[rd] & H^{\bullet}_{dR}(X;\mathbb{C}) \ar[d] & H^{\bullet,\bullet}_{\bar{\partial}}(X) \ar[ld] \\
 & H^{\bullet,\bullet}_{A}(X) &
}
\end{equation}
In general, the maps in (\ref{complex-cohom-diagram}) are neither surjective nor injective.
If the natural map
$$
H^{\bullet,\bullet}_{BC}(X)\rightarrow H^{\bullet,\bullet}_{dR}(X;\mathbb{C})
$$
induced by the identity is \emph{injective},
then we say that $X$ satisfies the \emph{$\partial\bar{\partial}$-Lemma}.
In addition, we have the following proposition (cf. \cite[Proposition 5.17]{DGMS75}).
\begin{prop}
Let $X$ be a compact complex manifold.
Then the following conditions are equivalent.
\begin{itemize}
  \item [(i)] The map $H^{\bullet,\bullet}_{BC}(X)\rightarrow H^{\bullet,\bullet}_{dR}(X;\mathbb{C})$ is injective.
  \item [(ii)] All the maps in \eqref{complex-cohom-diagram} are isomorphisms.
  \item [(iii)]The following equation holds on the space of $\mathbb{C}$-valued differential forms
  $$
  \ker\,\partial\cap\ker\,\bar{\partial}\cap\emph{im}\,d=\emph{im}\,\partial\bar{\partial}.
  $$
\end{itemize}
\end{prop}

If $X$ is a K\"{a}hler manifold or more general a $\partial\bar{\partial}$-manifold
then the Bott-Chern/Aeppli cohomology coincide with the Dolbeault cohomology via the maps in (\ref{complex-cohom-diagram}).
In non-K\"{a}hler complex geometry, the Bott-Chern/Aeppli cohomology play important roles.
According to the Hodge theory on compact Hermitian manifold we know that the de Rham and Dolbeault cohomology of a compact complex manifold are finite dimensional.
Analogously, Schweitzer \cite{Sch07} established the Hodge theory for Bott-Chern and Aeppli cohomology.
Given a Hermitian metric $h$ on $X$ then we can construct two $4$-th order elliptic self-adjoint differential operators
$$
\tilde{\Delta}_{BC}:=(\partial\bar{\partial})(\partial\bar{\partial})^{\ast}+(\partial\bar{\partial})^{\ast}(\partial\bar{\partial})
+(\bar{\partial}^{\ast}\partial)(\bar{\partial}^{\ast}\partial)^{\ast}
+(\bar{\partial}^{\ast}\partial)^{\ast}(\bar{\partial}^{\ast}\partial)
+\bar{\partial}^{\ast}\bar{\partial}+\partial^{\ast}\partial
$$
and
$$
\tilde{\Delta}_{A}:=\partial\partial^{\ast}+\bar{\partial}\,\bar{\partial}^{\ast}
+(\partial\bar{\partial})^{\ast}(\partial\bar{\partial})
+(\partial\bar{\partial})(\partial\bar{\partial})^{\ast}
+(\bar{\partial}\partial^{\ast})^{\ast}(\bar{\partial}\partial^{\ast})+(\bar{\partial}\partial^{\ast})
(\bar{\partial}\partial^{\ast})^{\ast}.
$$
Moreover, we have
$$
H^{\bullet,\bullet}_{BC}(X)\cong\ker\,\tilde{\Delta}_{BC}\,\,\,\mathrm{and}\,\,\, H^{\bullet,\bullet}_{A}\cong\ker\,\tilde{\Delta}_{A}.
$$
It follows that the Bott-Chern and Aeppli cohomology of a compact complex manifold are finite dimensional.
In addition, the Hodge star operator with respect to the Hermitian metric $h$ induces a duality
\begin{equation}\label{equ2.2}
\ast:H_{BC}^{p,q}(X) \stackrel{\cong}
 \longrightarrow H^{n-p,n-q}_{A}(X)
\end{equation}
for every $p,q\in\mathbb{N}$.

Consider the relationship between the de Rham and Dolbeault cohomology of $X$.
There exists a spectral sequence $\{E_{r},d_{r}\}$ of the double complex
$
(\mathcal{A}^{\bullet,\bullet}(X;\mathbb{C});\partial,\bar{\partial}),
$
called the \emph{Hodge-Fr\"{o}licher spectral sequence}
starting from the Dolbeault cohomology
$$
E^{p,q}_{1}=H^{p,q}_{\bar{\partial}}(X)
$$
and converging to the graded module associated to a filtration of the de Rham cohomology
$$
E^{p,q}_{r}\Longrightarrow \mathrm{Gr}\,H^{p+q}_{dR}(X;\mathbb{C})
\cong H^{p+q}_{dR}(X;\mathbb{C}).
$$
Moreover, this gives rise to the famous \emph{Fr\"{o}licher-inequality} \cite{Fro55}:
$$
\sum_{p+q=k}\mathrm{dim}_{\mathbb{C}}\,H^{p,q}_{\bar{\partial}}(X)\geq \mathrm{dim}_{\mathbb{C}}\,H^{k}_{dR}(X;\mathbb{C}),
$$
for any  $0\leq p,q\leq n$.
As an analogy, Angella-Tomassini \cite{AT13} proved a Fr\"{o}licher-type inequality
which links Bott-Chern and Aeppli cohomology to de Rham cohomology
and provides a new characteristic theorem of the $\partial\bar{\partial}$-Lemma.

\begin{thm}[{\cite[Theorem A and Theorem B]{AT13}}]\label{AT13}
Let $X$ be a compact complex manifold.
Then, for every $k\in\mathbb{N}$, the following inequality holds
\begin{equation}\label{fro-ine}
\sum_{p+q=k}\bigl(\mathrm{dim}_{\mathbb{C}}\,H_{BC}^{p,q}(X)+\mathrm{dim}_{\mathbb{C}}\,H^{p,q}_{A}(X)\bigr)
\geq2\mathrm{dim}_{\mathbb{C}}\,H^{k}_{dR}(X;\mathbb{C}).
\end{equation}
Moreover, the equality in \eqref{fro-ine} holds for every $k\in\mathbb{N}$ if and only if $X$ satisfies the $\partial\bar{\partial}$-Lemma.
\end{thm}

For any $(p,q)\in\mathbb{N}\times\mathbb{N}$, $k\in\mathbb{N}$ and $\sharp\in\{\partial,\bar{\partial}, BC, A\}$
we define
$$
h^{\bullet,\bullet}_{\sharp}(X):=\mathrm{dim}_{\mathbb{C}}\,H^{\bullet,\bullet}_{\sharp}(X)\quad
\mathrm{and}\quad
h^{k}_{\sharp}(X):=\sum_{p+q=k}\mathrm{dim}_{\mathbb{C}}\,H^{p,q}_{\sharp}(X).
$$
Especially, we denote the $k$-th Betti number of $X$ by
$$
b_{k}(X):=\mathrm{dim}_{\mathbb{C}}\,H^{k}_{dR}(X;\mathbb{C}).
$$
\begin{defn}[Non-K\"{a}hlerness degrees]
For a compact complex manifold $X$ of complex dimension $n$,
and for any $k\in\{0,1,\cdots,2n\}$ the $k$-th \emph{non-K\"{a}hlerness degree} is defined to be the non-negative integer
$$
\Delta^{k}(X):=h^{k}_{BC}(X)+h^{2n-k}_{BC}(X)-2b_{k}(X).
$$
In particular, we have $\Delta^{k}(X)=\Delta^{2n-k}(X)$ and $\Delta^{0}=\Delta^{2n}=0$; therefore, the only non-trivial non-K\"{a}hlerness degrees are $\Delta^{k}(X)$ for $k\in \{1, 2, \cdots 2n-1\}$.
\end{defn}

Equivalently, a quantitative characterization of the $\partial\bar{\partial}$-Lemma on compact complex manifolds is stated as follows.
\begin{thm}[{\cite[Theorem 3.1]{AT17}}]\label{thm2.4}
A compact complex manifold $X$ satisfies the $\partial\bar{\partial}$-Lemma if and only if, for any $k\in\mathbb{N}$, there holds
\begin{equation*}
\sum_{p+q=k}(h_{BC}^{p,q}(X)-h^{p,q}_{A}(X))=0.
\end{equation*}
\end{thm}
The \emph{revised non-K\"{a}hlerness degrees} of $X$ is defined to be the following integers:
$$
N^{k}(X):=h^{k}_{BC}(X)-h^{k}_{A}(X)=h^{k}_{BC}(X)-h^{2n-k}_{BC}(X)\in\mathbb{Z},
$$
where $k\in\{0,1,\cdots,2n\}$.
From definition, we get $N^{k}(X)=-N^{2n-k}(X)$;
moreover, we have the following equivalent conditions (cf. \cite{AT13,AT17}):
\begin{itemize}
  \item [(i)] $X$ satisfies the $\partial\bar{\partial}$-Lemma;
  \item [(ii)] $N^{k}(X)=0$ for any $k\in\{0,1,\cdots,2n\}$;
  \item [(iii)] $\Delta^{k}(X)=0$ for any $k\in\{0,1,\cdots,2n\}$.
\end{itemize}

From the bimeromorphic geometry point of view, it would be interesting to confirm whether the non-K\"{a}hlerness degrees of a compact complex manifold are \emph{bimeromorphic invariants}.
It is noteworthy that for any compact complex surface the non-K\"{a}hlerness degree $\Delta^{1}$ is always zero and $\Delta^{2}=0$ if and only if the surface is K\"{a}hler (cf. \cite{AGT16}).
In addition, if the surface is non-K\"{a}hler then $\Delta^{2}=2$ (cf. \cite{ATV16,Tel06}).
\subsection{Bott-Chern hypercohomology}\label{sec2.2}
Assume that $(X,J)$ is a compact complex manifold with $\mathrm{dim}_{\mathbb{C}}\,X=n$.
Let $\mathscr{E}^{s,t}_{X}$ be the sheaf of germs of differential $(s,t)$-forms on $X$.

For a fixed bidegree $(p,q)$ such that $q\geq p\geq 1$,
let $k=p+q$.
We can define the following sheaves
$$
\mathscr{L}^{l}_{X}=\bigoplus_{s+t=l,\atop s<p,t<q}\mathscr{E}^{s,t}_{X}\,\,\,\mathrm{for}\,\,\,l\leq k-2
$$
and
$$
\mathscr{L}^{l-1}_{X}=\bigoplus_{s+t=l,\atop s\geq p,t\geq q}\mathscr{E}^{s,t}_{X}\,\,\,\mathrm{for}\,\,\,l\geq k.
$$
Observe that
$
\mathscr{L}^{k-2}_{X}=\mathscr{E}^{p-1,q-1}_{X}
$
and
$
\mathscr{L}^{k-1}_{X}=\mathscr{E}^{p,q}_{X}.
$
We define the operators acting on $\mathscr{L}^{\bullet}_{X}$ by setting
$$
\delta^{l}:=d:\mathscr{L}^{l}_{X}\rightarrow\mathscr{L}^{l+1}_{X},
$$
for $l\neq k-2$ (in the case of $l\leq k-3$, we neglect the components which fall outside $\mathscr{L}^{l+1}_{X}$),
and
$$
\delta^{k-2}:=\partial\bar{\partial}:\mathscr{L}^{k-2}_{X}\rightarrow\mathscr{L}^{k-1}_{X}.
$$
Since $\delta^{i+1}\circ\delta^{i}=0$ for any $i\geq0$,
we get a sheaf complex $\mathscr{L}_{X}^{\bullet}$ as follows
\begin{equation*}
\xymatrix{
\cdots \ar[r]^{d} & \mathscr{L}^{k-3}_{X} \ar[r]^{d} & \mathscr{L}^{k-2}_{X}
  \ar[r]^{\partial\bar{\partial}}&\mathscr{L}^{k-1}_{X} \ar[r]^{d} &  \mathscr{L}^{k}_{X} \ar[r]^{d}
  & \cdots .}
\end{equation*}

From definition, we have
$
H^{p,q}_{BC}(X)\cong H^{k-1}(\mathscr{L}^{\bullet}_{X}(X)).
$
The $(p,q)$-\emph{Bott-Chern hypercohomology} of $X$ is defined to be the $(k-1)$-th hypercohomology of $\mathscr{L}_{X}^{\bullet}$, i.e.,
$$
\mathbb{H}^{k-1}(X,\mathscr{L}_{X}^{\bullet}):=R^{k-1}\Gamma(X, \mathscr{L}_{X}^{\bullet}),
$$
where $R^{k-1}\Gamma(X, -)$ is the right derived functor of the global section functor $\Gamma(X, -)$.
Because $\mathscr{L}^{\bullet}_{X}$ is a complex of fine sheaves the following lemma holds (cf. \cite[Chapter VI, \S 12]{Dem12}).

\begin{lem}\label{lem2.4}
Let $X$ be a compact complex manifold.
Then for the fixed bi-degree $(p,q) $ we have
\begin{equation*}
H_{BC}^{p,q}(X)\cong\mathbb{H}^{p+q-1}(X,\mathscr{L}_{X}^{\bullet}).
\end{equation*}
\end{lem}

Let $\Omega_{X}^{l}$ (resp. $\bar{\Omega}_{X}^{l}$)
be the sheaf of holomorphic (resp. anti-holomorphic) $p$-forms and
$\mathbb{C}_{X}$ the constant sheaf on $X$.
There exist two sheaf complexes of holomorphic and anti-holomorphic forms
\begin{equation*}
\xymatrix@C=0.4cm{
0\ar[r]^{} &\mathcal{O}_{X}\ar[r]^{\partial} &\Omega_{X}^{1} \ar[r]^{\partial}& \cdots  \ar[r]^{\partial}& \Omega_{X}^{p-1}  \ar[r]^{}& 0,}\;\;
\xymatrix@C=0.4cm{
0\ar[r]^{} &\bar{\mathcal{O}}_{X}\ar[r]^{\bar{\partial}} & \bar{\Omega}_{X}^{1} \ar[r]^{\bar{\partial}}& \cdots  \ar[r]^{\bar{\partial}}& \bar{\Omega}_{X}^{q-1}  \ar[r]^{}& 0;}
\end{equation*}
moreover, we can define a new sheaf complex $\mathscr{S}_{X}^{\bullet}$:
\begin{equation*}
\xymatrix@C=0.4cm{
0\ar[r]^{} &\mathcal{O}_{X}+\bar{\mathcal{O}}_{X} \ar[r]^{\partial+\bar{\partial}} &\Omega_{X}^{1}\oplus \bar{\Omega}_{X}^{1} \ar[r]^{}& \cdots  \ar[r]^{}& \Omega_{X}^{p-1}\oplus \bar{\Omega}_{X}^{p-1} \ar[r]^{} & \bar{\Omega}_{X}^{p} \ar[r]^{} & \cdots  \bar{\Omega}_{X}^{q-1} \ar[r]^{} & 0.}
\end{equation*}
Similarly, we have a sheaf complex $\mathscr{B}_{X}^{\bullet}$:
\begin{equation}\label{equal-BC-complex}
\xymatrix@C=0.4cm{
0\ar[r]^{} & \mathbb{C}_{X}\ar[r]^{(+,-)\;\;\;\;\;\;\;} & \mathcal{O}_{X}\oplus \bar{\mathcal{O}}_{X} \ar[r]^{} &\Omega_{X}^{1}\oplus \bar{\Omega}_{X}^{1}
  \ar[r]^{}& \cdots  \ar[r]^{}& \Omega_{X}^{p-1}\oplus \bar{\Omega}_{X}^{p-1} \ar[r]^{} & \bar{\Omega}_{X}^{p} \ar[r]^{} & \cdots  \bar{\Omega}_{X}^{q-1} \ar[r]^{} & 0.}
\end{equation}

It is of importance that the inclusion $\mathscr{S}_{X}^{\bullet}\subset \mathscr{L}_{X}^{\bullet}$ and the morphism
$\mathscr{B}_{X}^{\bullet} \rightarrow \mathscr{S}_{X}^{\bullet}[-1]$
are all quasi-isomorphic, where $\mathscr{S}_{X}^{\bullet}[-1]:=\mathscr{S}_{X}^{\bullet-1}$ (cf. \cite[Lemma 12.1]{Dem12} or \cite[Proposition 4.2, 4.3]{Sch07}).
Consequently, by Lemma \ref{lem2.4} we get the following isomorphisms.

\begin{prop}\label{BC-char}
$H_{BC}^{p,q}(X)
\cong
\mathbb{H}^{p+q-1}(X,\mathscr{L}_{X}^{\bullet})
\cong
\mathbb{H}^{p+q-1}(X,\mathscr{S}_{X}^{\bullet})
\cong
\mathbb{H}^{p+q}(X,\mathscr{B}_{X}^{\bullet})
$.
\end{prop}


\section{Blow-up formula of Bott-Chern cohomology}\label{blow-up-formula-BC}
The purpose of this section is to give the proof of Theorem \ref{Thm1}.
Throughout of this section we assume that $X$ is a compact complex manifold with $\textmd{dim}_{\mathbb{C}}X=n$
and $\iota:Z\hookrightarrow X$ be a closed complex submanifold with $\textmd{codim}_{\mathbb{C}}Z=r\geq 2$.
Set $U=X-Z$.

\subsection{Case of $1\leq p,q\leq n$}
Due to the symmetric property of Bott-Chern cohomology we only need to consider the case of $q\geq p\geq1$.
Recall that the space of sections of the constant sheaf $\mathbb{C}_{X}$ over an open subset $V$ is
$$
\Gamma(V,\mathbb{C}_{X})=\{f:V\rightarrow \mathbb{C}\; \textrm{locally constant function}\}.
$$
We define a surjective morphism of sheaves $\chi: \mathbb{C}_{X}\rightarrow \iota_{\ast}\mathbb{C}_{Z}$ by setting
\begin{eqnarray*}
  \chi(V):\Gamma(V,\mathbb{C}_{X})&\rightarrow&
  \Gamma(V,\iota_{\ast}\mathbb{C}_{Z})=\Gamma(V\cap Z, \mathbb{C}_{Z}) \\
  f &\mapsto& f|_{V\cap Z},
\end{eqnarray*}
for any open subset $V\subset X$ such that $V\cap Z \neq \emptyset$, and {\it zero map} otherwise.
Then there is a short exact sequence of sheaves
\begin{equation}\label{pre-exact-seq-C}
\xymatrix{
  0 \ar[r] & \mathcal{K}_{U} \ar[r]^{} & \mathbb{C}_{X} \ar[r]^{\chi} & \iota_{\ast}\mathbb{C}_{Z} \ar[r] & 0}
\end{equation}
where $\mathcal{K}_{U}$ is the kernel sheaf of $\chi$.

For any integer $0\leq s\leq n$ and any open subset $V\subset X$ we can define a morphism
\begin{eqnarray*}
  \varphi(V):\Gamma(V,\Omega_{X}^{s})&\rightarrow&
  \Gamma(V,\iota_{\ast}\Omega^{s}_{Z})=\Gamma(V\cap Z, \Omega^{s}_{Z}) \\
  \alpha &\mapsto&(\iota_{V\cap Z})^{\ast} \alpha,
\end{eqnarray*}
where $\iota_{V\cap Z}:V\cap Z\rightarrow V$ is the holomorphic inclusion and therefore we get a sheaf morphism
$$
\varphi:\Omega_{X}^{s}\rightarrow \iota_{\ast}\Omega^{s}_{Z}.
$$
Set $\mathcal{K}_{X}^{s}=\ker\,(\varphi)$.
Then we have the following sequence of sheaves;
for $s=0$ it is the structure sheaf sequence.

\begin{lem}
The following short sequence of sheaves on $X$ is exact
\begin{equation}\label{Z-pre-exact-seq}
\xymatrix{
  0 \ar[r] & \mathcal{K}^{s}_{X} \ar[r]^{} & \Omega_{X}^{s} \ar[r]^{\varphi} & \iota_{\ast}\Omega_{Z}^{s} \ar[r] & 0.}
\end{equation}
\end{lem}

\begin{proof}
We only need to show that the sheaf morphism $\varphi$ is surjective.
By definition, the stalk of the direct image $\iota_{\ast}\Omega_{Z}^{s}$ at a point $x$ is $(\Omega_{Z}^{s})_{x}$ if $x\in Z$ and zero otherwise.
It follows that the map of stalks
$\varphi_{x}:(\Omega^{s}_{X})_{x}\rightarrow(\iota_{\ast}\Omega^{s}_{Z})_{x}$ is zero map if $x\in U$.
The remaining thing to do in the proof is to verify that for any $x\in Z$ the morphism of stalks
$$
\varphi_{x}:(\Omega^{s}_{X})_{x}\rightarrow(\iota_{\ast}\Omega^{s}_{Z})_{x}
$$
is surjective.
Let $(W; z_{1},z_{2},\cdots,z_{n})$ be a local coordinate chart of $x$ such that
$$W\cap Z=\{z_{n-r+1}=\cdots=z_{n}=0\}.$$
Then there exists a holomorphic projection given by
\begin{eqnarray*}
\tau:W &\rightarrow& W\cap Z, \\
(z_{1},\cdots,z_{n}) &\mapsto& (z_{1},\cdots,z_{n-r}).
\end{eqnarray*}
For any germ $\alpha_{x} \in (\iota_{\ast}\Omega_{Z}^{s})_{x}$ we can choose a representative $(V,\alpha)$,
here $\alpha$ is a holomorphic $p$-form on $V\cap Z$.
Particularly, we can choose $V$ small enough such that $V\subset W$ and then the restriction of $\tau$ on $V$ gives rise to a holomorphic map
$\tau_{V}:V\rightarrow V\cap Z$
satisfying $\tau_{V}\circ \iota_{V\cap Z}=\textmd{id}_{V\cap Z}$.
Let $\beta=(\tau_{V})^{\ast}(\alpha)$, then $(V,\beta)$ represents a germ, denoted by $\beta_{x}$, in the stalk $(\Omega^{s}_{X})_{x}$.
From definition, we get
$$
(\iota_{V\cap Z})^{\ast}\beta=(\iota_{V\cap Z})^{\ast}\bigl((\tau_{V})^{\ast}(\alpha)\bigr)=(\tau_{V}\circ\iota_{V\cap Z})^{\ast}(\alpha)=\alpha.
$$
It follows that $\varphi_{x}(\beta_{x})=\alpha_{x}$, namely, $\varphi_{x}$ is surjective.
\end{proof}

Taking the complex conjugation of \eqref{Z-pre-exact-seq},
we get the the anti-holomorphic sheaf sequence
\begin{equation}\label{Z-pre-exact-seq-conju}
\xymatrix{
  0 \ar[r] &  \bar{\mathcal{K}}^{t}_{X} \ar[r]^{} &\bar{\Omega}_{X}^{t} \ar[r]^{\bar{\varphi}} & \iota_{\ast}\bar{\Omega}_{Z}^{t} \ar[r] & 0,}
\end{equation}
for any integer $0\leq t\leq n$.
Since the direct image functor $\iota_{\ast}$ commutes with the direct sum we derive a short exact sequence by taking direct sum of \eqref{Z-pre-exact-seq} and \eqref{Z-pre-exact-seq-conju}
\begin{equation}\label{sheaf-sum-exact}
\xymatrix{
  0 \ar[r] &  \mathcal{K}^{s}_{X}\oplus\bar{\mathcal{K}}^{t}_{X} \ar[r]^{} & \Omega_{X}^{s}\oplus \bar{\Omega}_{X}^{t} \ar[r]^{\varphi\oplus\bar{\varphi}\quad} & \iota_{\ast}(\Omega_{Z}^{s}\oplus \bar{\Omega}_{Z}^{t}) \ar[r] & 0.}
\end{equation}
Observing that both operators $\partial$ and $\bar{\partial}$ commute with the pullback of holomorphic maps and $\iota_{\ast}$ is an exact functor,
we get the following commutative diagram from \eqref{sheaf-sum-exact}
\begin{equation}\label{commu-sheaf-sum-exact}
\xymatrix@C=0.5cm{
0 \ar[r]^{} & \mathcal{K}^{s}_{X}\oplus\bar{\mathcal{K}}^{t}_{X}  \ar[d]_{\partial\oplus \bar{\partial}} \ar[r]^{} &  \Omega_{X}^{s}\oplus \bar{\Omega}_{X}^{t}\ar[d]_{\partial\oplus \bar{\partial}} \ar[r]^{} & \iota_{\ast} (\Omega_{Z}^{s}\oplus \bar{\Omega}_{Z}^{t})\ar[d]_{\partial\oplus \bar{\partial}} \ar[r]^{} & 0 \\
0 \ar[r] & \mathcal{K}^{s+1}_{X}\oplus\bar{\mathcal{K}}^{t+1}_{X} \ar[r]^{} &  \Omega_{X}^{s+1}\oplus \bar{\Omega}_{X}^{t+1} \ar[r]^{} &
\iota_{\ast} (\Omega_{Z}^{s+1}\oplus \bar{\Omega}_{Z}^{t+1})\ar[r]^{} & 0.}
\end{equation}
By the definition of the sheaf complex \eqref{equal-BC-complex},
the sequences \eqref{pre-exact-seq-C} and \eqref{commu-sheaf-sum-exact} implies

\begin{prop}
For $q\geq p\geq 1$,
there is a short exact sequence of sheaf complexes
\begin{equation}\label{Z-holomor-seq}
\xymatrix{
  0 \ar[r] & \mathscr{K}_{X}^{\bullet} \ar[r]^{} & \mathscr{B}_{X}^{\bullet} \ar[r]^{} & \iota_{\ast}\mathscr{B}_{Z}^{\bullet} \ar[r] & 0,}
\end{equation}
where $\mathscr{K}_{X}^{\bullet}$ is the sheaf complex
\begin{equation}\label{kernel-complex}
\xymatrix@C=0.4cm{
\mathcal{K}_{U} \ar[r]^{(+,-)\;\;\;\;\;\;\;} & \mathcal{K}_{X}^{0}\oplus \bar{\mathcal{K}}_{X}^{0} \ar[r]^{} &\mathcal{K}_{X}^{1}\oplus \bar{\mathcal{K}}_{X}^{1}
  \ar[r]^{}& \cdots  \ar[r]^{}& \mathcal{K}_{X}^{p-1}\oplus \bar{\mathcal{K}}_{X}^{p-1} \ar[r]^{} & \bar{\mathcal{K}}_{X}^{p} \ar[r]^{} & \cdots  \bar{\mathcal{K}}_{X}^{q-1} \ar[r]^{} & 0.}
\end{equation}
\end{prop}

To prove the theorem we need the following basic result from homological algebra.
\begin{prop}\label{cla}
Consider a commutative diagram of two exact sequences of complex vector spaces
\begin{equation*}
\xymatrix@C=0.5cm{
  \cdots \ar[r]^{} & A_1 \ar[d]_{i_1} \ar[r]^{f_1}& A_2 \ar[d]_{i_2} \ar[r]^{f_2}& A_3 \ar[d]_{i_3} \ar[r]^{f_3}& A_4 \ar[d]_{i_4} \ar[r]^{}& \cdots \\
\cdots \ar[r]^{} & B_1 \ar[r]^{g_1}& B_2 \ar[r]^{g_2}& B_3 \ar[r]^{g_3}& B_4 \ar[r]^{}&  \cdots.}
\end{equation*}
Assume that $i_1,\,i_4$ are isomorphic, and $i_2,\,i_3$ are injective.
Then there is a natural isomorphism
$$
B_2\cong A_2\oplus \Big(B_3/ i_3(A_3)\Big).
$$
\end{prop}
\begin{proof}
By the hypothesis, $i_2$ and $i_3$ are injective and therefore we have
$$
B_2\cong A_2\oplus \mathrm{coker}\, i_2,\,\,\, B_3\cong A_3\oplus \mathrm{coker}\,i_3.
$$
The commutativity of the second square implies $g_{2}(\mathrm{im}(i_{2}))\subset\mathrm{im}(i_{3})$ and therefore there exists a well-defined morphism of cokernels
$$
\bar{g}_{2}:\mathrm{coker}\,i_2\rightarrow\mathrm{coker}\,i_3.
$$
The injectivity and the surjectivity of $\bar{g}_{2}$ can be obtained by the standard argument of diagram-chasing.
\end{proof}

We are now in a position to prove the first part of Theorem \ref{Thm1}.
\begin{proof}[Proof of Theorem \ref{Thm1} $\mathrm{(i)}$]
Recall that $X$ is a compact complex manifold with $\mathrm{dim}_{\mathbb{C}}\,X=n$,
$\iota: Z\hookrightarrow X$ a closed complex submanifold such that $\mathrm{codim}_{\mathbb{C}}Z=r\geq2$ and $U=X-Z$ the complimentary open set.
Assume that $\pi:\tilde{X}\rightarrow X$ is the blow-up of $X$ with the center $Z$.
Let $E:=\pi^{-1}(Z)$ be the exceptional divisor and $\tilde{U}:=\tilde{X}-E$.
Then the restriction map
$$
\pi|_{\tilde{U}}: \tilde{U} \rightarrow U
$$
is biholomorphic.
Moreover, we have a commutative diagram
\begin{equation}\label{b-l-d}
\xymatrix{
E \ar[d]_{\pi|_{E}} \ar@{^{(}->}[r]^{\tilde{\iota}} & \tilde{X}\ar[d]^{\pi}\\
 Z \ar@{^{(}->}[r]^{\iota} & X,
}
\end{equation}
where $\tilde{\iota}: E \hookrightarrow \tilde{X}$ is the inclusion.

Due to the symmetric property of Bott-Chern cohomology we only consider the case of $q\geq p\geq1$.
For any bi-degree $(p,q)$ satisfying $q\geq p\geq1$ we can construct the following short exact sequence of sheaves on $\tilde{X}$ similar to \eqref{Z-holomor-seq}
\begin{equation}\label{E-holomor-seq}
\xymatrix{
  0 \ar[r] & \mathscr{K}^{\bullet}_{\tilde{X}} \ar[r]^{} & \mathscr{B}_{\tilde{X}}^{\bullet} \ar[r]^{} & \tilde{\iota}_{\ast}\mathscr{B}_{E}^{\bullet} \ar[r] & 0,}
\end{equation}
where $\mathscr{B}_{\tilde{X}}$ and $\mathscr{B}_{E}$ are defined as \eqref{equal-BC-complex},
and $\mathscr{K}^{\bullet}_{\tilde{X}}$ is the counterpart of \eqref{kernel-complex} on $\tilde{X}$.

Note that $\pi:\tilde{X}\rightarrow X$ is a proper holomorphic map the pullback $\pi^{\ast}$ of holomorphic forms induces a morphism (over $\pi$) from $\Omega_{X}^{s}$ to $\Omega_{\tilde{X}}^{s}$,
and also a morphism from $\bar{\Omega}_{X}^{s}$ to $\bar{\Omega}_{\tilde{X}}^{s}$;
moreover, the blow-up diagram \eqref{b-l-d} gives rise to a commutative diagram for the long exact sequences of hypercohomologies associated to \eqref{Z-holomor-seq} and \eqref{E-holomor-seq}, respectively.
\begin{equation}\label{comm-long-seq-0}
\xymatrix@C=0.4cm{
  \cdots \ar[r]^{}
  & \mathbb{H}^{p+q}(X, \mathscr{K}_{X}^{\bullet})\ar[d]_{\pi^{\ast}} \ar[r]^{}
  & \mathbb{H}^{p+q}(X, \mathscr{B}_{X}^{\bullet})\ar[d]_{\pi^{\ast}} \ar[r]^{}
  & \mathbb{H}^{p+q}(X, \iota_{\ast}\mathscr{B}_{Z}^{\bullet}) \ar[d]_{(\pi|_{E})^{\ast}} \ar[r]^{}
  & \mathbb{H}^{p+q+1}(X, \mathscr{K}_{X}^{\bullet}) \ar[d]_{\pi^{\ast}} \ar[r]^{} & \cdots \\
   \cdots \ar[r]
  & \mathbb{H}^{p+q}(\tilde{X}, \mathscr{K}_{\tilde{X}}^{\bullet})\ar[r]^{}
  & \mathbb{H}^{p+q}(\tilde{X}, \mathscr{B}_{\tilde{X}}^{\bullet}) \ar[r]^{}
  & \mathbb{H}^{p+q}(\tilde{X}, \tilde{\iota}_{\ast}\mathscr{B}_{E}^{\bullet}) \ar[r]
  & \mathbb{H}^{p+q+1}(\tilde{X}, \mathscr{K}_{\tilde{X}}^{\bullet})\ar[r]&\cdots}
\end{equation}

Because the direct image functor $\iota_{\ast}$ is exact
and $R\Gamma(X,\iota_{\ast}(-))=R\Gamma(Z,-)$, by Proposition \ref{BC-char} we obtain
$$
\mathbb{H}^{p+q}(X, \iota_{\ast}\mathscr{B}_{Z}^{\bullet})
\cong
\mathbb{H}^{p+q}(Z, \mathscr{B}_{Z}^{\bullet})
\cong
H_{BC}^{p,q}(Z),
$$
also by Proposition \ref{BC-char} we have $\mathbb{H}^{p+q}(X, \mathscr{B}_{X}^{\bullet})\cong H_{BC}^{p,q}(X)$.
As a result, \eqref{comm-long-seq-0} becomes
\begin{equation}\label{comm-cohom-long-seq}
\xymatrix@C=0.4cm{
  \cdots \ar[r]^{}
  & \mathbb{H}^{p+q}(X, \mathscr{K}_{X}^{\bullet})\ar[d]_{\pi^{\ast}} \ar[r]^{}
  & H_{BC}^{p,q}(X)\ar[d]_{\pi^{\ast}} \ar[r]^{}
  & H_{BC}^{p,q}(Z) \ar[d]_{(\pi|_{E})^{\ast}} \ar[r]^{}
  & \mathbb{H}^{p+q+1}(X, \mathscr{K}_{X}^{\bullet}) \ar[d]_{\pi^{\ast}} \ar[r]^{} & \cdots \\
   \cdots \ar[r]
  & \mathbb{H}^{p+q}(\tilde{X}, \mathscr{K}_{\tilde{X}}^{\bullet})\ar[r]^{}
  & H_{BC}^{p,q}(\tilde{X}) \ar[r]^{}
  & H_{BC}^{p,q}(E) \ar[r]
  & \mathbb{H}^{p+q+1}(\tilde{X}, \mathscr{K}_{\tilde{X}}^{\bullet})\ar[r]&\cdots}
\end{equation}

We claim the following result and give its proof at the end of this subsection.

\begin{lem}\label{kersheaf-Hcoh-iso}
For any $k\in \mathbb{N}$, the morphism
$$
\pi^{\ast}:\mathbb{H}^{k}(X, \mathscr{K}_{X}^{\bullet})\rightarrow\mathbb{H}^{k}(\tilde{X}, \mathscr{K}_{\tilde{X}}^{\bullet})
$$
is isomorphic.
\end{lem}

From \eqref{comm-cohom-long-seq} and Lemma \ref{kersheaf-Hcoh-iso}, we derive the commutative diagram
\begin{equation}\label{final-cohom-long-seq}
\xymatrix@C=0.4cm{
  \cdots \ar[r]^{}
  & \mathbb{H}^{p+q}(X, \mathscr{K}_{X}^{\bullet})\ar[d]_{\pi^{\ast}}^{\cong} \ar[r]^{}
  & H_{BC}^{p,q}(X)\ar[d]_{\pi^{\ast}} \ar[r]^{}
  & H_{BC}^{p,q}(Z) \ar[d]_{(\pi|_{E})^{\ast}} \ar[r]^{}
  & \mathbb{H}^{p+q+1}(X, \mathscr{K}_{X}^{\bullet}) \ar[d]_{\pi^{\ast}}^{\cong}  \ar[r]^{} & \cdots \\
   \cdots \ar[r]
  & \mathbb{H}^{p+q}(\tilde{X}, \mathscr{K}_{\tilde{X}}^{\bullet})\ar[r]^{}
  & H_{BC}^{p,q}(\tilde{X}) \ar[r]^{}
  & H_{BC}^{p,q}(E) \ar[r]
  & \mathbb{H}^{p+q+1}(\tilde{X}, \mathscr{K}_{\tilde{X}}^{\bullet})\ar[r]&\cdots.}
\end{equation}
Since $\pi:\tilde{X}\rightarrow X$ is a proper and holomorphic map the induced morphism
\begin{equation}\label{inj-X}
\pi^{\ast}: H_{BC}^{p,q}(X)\rightarrow H_{BC}^{p,q}(\tilde{X})
\end{equation}
is injective (cf. \cite[Theorem 12.9]{Dem12} or \cite[Theorem 2.4.]{AK14}) and so is the morphism
\begin{equation}\label{inj-ZE}
(\pi|_{E})^{\ast}:H_{BC}^{p,q}(Z)\rightarrow H_{BC}^{p,q}(E)
\end{equation}
by the Weak Five Lemma \cite[Chapter 1, \S 3]{Mac63}.
Applying Proposition \ref{cla} to \eqref{final-cohom-long-seq}, we obtain
$$
\mathrm{coker}\, \eqref{inj-X} \cong \mathrm{coker}\,\eqref{inj-ZE},
$$
which means
\begin{equation*}
H_{BC}^{p, q}(\tilde{X})
\cong
H_{BC}^{p, q}(X)\oplus \Big(H_{BC}^{p, q}(E)/(\pi|_{E})^{\ast}H_{BC}^{p, q}(Z)\Big).
\end{equation*}
The proof of Theorem \ref{Thm1} (i) is now complete.
\end{proof}

\begin{rem}\label{Cech-cohom-rem}
Note that the hypercohomology groups of a sheaf complex are canonically isomorphic to its \v{C}ech hypercohomology groups on manifolds (cf. \cite[Theorem 1.3.13 (3)]{Bry93}).
The diagram \eqref{comm-long-seq-0} is actually equal to a diagram of
\v{C}ech hypercohomology groups whose commutativity can be also verified directly,
because all morphisms are induced by the pullback of differential forms.
\end{rem}

Finally, we turn to the proof of Lemma \ref{kersheaf-Hcoh-iso}.

For the pair $(X,Z)$,
there exist two surjective sheaf morphisms
$
\chi: \mathbb{C}_{X}\rightarrow \iota_{\ast}\mathbb{C}_{Z}$
and
$\nu: \E_{X}^{l} \rightarrow \iota_{\ast}\E_{Z}^{l}$ defined by setting
\begin{equation*}
  \Gamma(V,\mathbb{C}_{X})
  \rightarrow
  \Gamma(V,\iota_{\ast}\mathbb{C}_{Z}),\,\,\,
  f \mapsto f|_{V\cap Z},
\end{equation*}
and
\begin{equation*}
\Gamma(V, \E_{X}^{l})
\rightarrow
\Gamma(V,\iota_{\ast}\E_{Z}^{l} ),\,\,\,
\alpha \mapsto (\iota_{V\cap Z})^{\ast} \alpha,
\end{equation*}
for any open subset $V\subset X$ such that $V\cap Z \neq \emptyset$, and {\it zero map} otherwise.
Let $\K_{U}$ and $\K_{U}^{l}$ are the kernel sheaves of $\chi$ and $\nu$, respectively.
It can be directly verified that $\K_{U}^{\bullet}$ is a \emph{fine} resolution of $\K_{U}$,
and hence the sheaf cohomology of $\K_{U}$ is isomorphic to the cohomology of the complex $(\Gamma(X,\K_{U}^{\bullet}),d)$ called the \emph{relative de Rham cohomology}.
Moreover, by \cite[Proposition 13.11]{MT97}, we have
\begin{equation*}\label{rel-deRham}
H^{l}(X, \K_{U})
\cong
H_{dR}^{l}(X, Z)
\cong
H_{dR, c}^{l}(U),
\end{equation*}
where $H_{dR}^{l}(X, Z)$ is the $l$-th \emph{relative de Rham cohomology} in the sense of Godbillon 
\cite[Chapitre XII]{Go71}
and $H_{dR, c}^{l}(U)$ is the $l$-th compactly supported de Rham cohomology of $U$.

\begin{proof}[Proof of Lemma \ref{kersheaf-Hcoh-iso}]
Recall that the sheaf complex $\mathscr{K}^{\bullet}_{X}$ is defined to be
\begin{equation*}
\xymatrix@C=0.4cm{
\mathcal{K}_{U} \ar[r]^{(+,-)\;\;\;\;\;\;\;} & \mathcal{K}_{X}^{0}\oplus \bar{\mathcal{K}}_{X}^{0} \ar[r]^{} &\mathcal{K}_{X}^{1}\oplus \bar{\mathcal{K}}_{X}^{1}
  \ar[r]^{}& \cdots  \ar[r]^{}& \mathcal{K}_{X}^{p-1}\oplus \bar{\mathcal{K}}_{X}^{p-1} \ar[r]^{} & \bar{\mathcal{K}}_{X}^{p} \ar[r]^{} & \cdots  \bar{\mathcal{K}}_{X}^{q-1} \ar[r]^{} & 0.}
\end{equation*}
The sheaf $\mathcal{K}_{U}$ can be viewed as a complex concentrated in degree $0$.
Denote $\mathscr{D}^{\bullet}_{X}$ be the sheaf complex
\begin{equation*}
\xymatrix@C=0.4cm{
0\ar[r]^{} & \mathcal{K}_{X}^{0}\oplus \bar{\mathcal{K}}_{X}^{0} \ar[r]^{} &\mathcal{K}_{X}^{1}\oplus \bar{\mathcal{K}}_{X}^{1}
  \ar[r]^{}& \cdots  \ar[r]^{}& \mathcal{K}_{X}^{p-1}\oplus \bar{\mathcal{K}}_{X}^{p-1} \ar[r]^{} & \bar{\mathcal{K}}_{X}^{p} \ar[r]^{} & \cdots  \bar{\mathcal{K}}_{X}^{q-1} \ar[r]^{} & 0.}
\end{equation*}
Note that the following diagram is commutative
\begin{equation*}
\small{
\xymatrix@C=0.5cm{
   & \vdots  & \vdots  & \vdots & \\
0\ar[r] & \K_{X}^{1}\oplus \bar{\K}_{X}^{1}\ar[u] \ar[r]^{\Id} &\K_{X}^{1}\oplus \bar{\K}_{X}^{1}  \ar[u]\ar[r] & 0 \ar[u]\ar[r] &0 \\
0\ar[r] &  \K_{X}^{0}\oplus \bar{\K}_{X}^{0}  \ar[u]^{\partial\oplus \bar{\partial}} \ar[r]^{\Id} & \K_{X}^{0}\oplus \bar{\K}_{X}^{0}  \ar[u]^{\partial\oplus \bar{\partial}}\ar[r] & 0 \ar[u]\ar[r] &0 \\
0\ar[r] & 0  \ar[u] \ar[r] & \K_{U}   \ar[u]^{(+,1)}\ar[r]^{\Id} & \K_{U} \ar[u] \ar[r] &0 .\\
& 0 \ar[u]   & 0 \ar[u]  & 0\ar[u] &}}
\end{equation*}
We obtain a short exact sequence of sheaf complexes
\begin{equation*}
\xymatrix{
0 \ar[r] & \mathscr{D}_{X}^{\bullet}[-1]
\ar[r] & \mathscr{K}_{X}^{\bullet}
\ar[r] &  \mathcal{K}_{U}  \ar[r] & 0,}
\end{equation*}
and hence a long exact sequence of hypercohomologies:
\begin{equation*}
\xymatrix@C=0.4cm{
\ar[r]^{}
& H^{l-1}(X,\mathcal{K}_{U}) \ar[r]^{}
& \mathbb{H}^{l}(X,\mathscr{D}_{X}^{\bullet}[-1]) \ar[r]^{}
& \mathbb{H}^{l}(X, \mathscr{K}_{X}^{\bullet})  \ar[r]^{}
& H^{k}(X,\mathcal{K}_{U})  \ar[r]^{}
& \mathbb{H}^{l+1}(X,\mathscr{D}_{X}^{\bullet}[-1])
\ar[r]^{} &}
\end{equation*}
which equals to
\begin{equation*}
\xymatrix@C=0.4cm{
\cdots\ar[r]^{}
& H^{l-1}(X,\mathcal{K}_{U}) \ar[r]^{}
& \mathbb{H}^{l-1}(X,\mathscr{D}_{X}^{\bullet}) \ar[r]^{}
& \mathbb{H}^{l}(X, \mathscr{K}_{X}^{\bullet})  \ar[r]^{}
& H^{l}(X,\mathcal{K}_{U})  \ar[r]^{}
& \mathbb{H}^{l}(X,\mathscr{D}_{X}^{\bullet})
\ar[r]^{} &\cdots}
\end{equation*}
Similarly, for the sheaf complex $\mathscr{K}^{\bullet}_{\tilde{X}}$ with respect to the pair $(\tilde{X},E)$, there is a long exact sequence
\begin{equation*}
\xymatrix@C=0.4cm{
\cdots\ar[r]^{}
& H^{l-1}(\tilde{X},\mathcal{K}_{\tilde{U}}) \ar[r]^{}
& \mathbb{H}^{l-1}(\tilde{X},\mathscr{D}_{\tilde{X}}^{\bullet}) \ar[r]^{}
& \mathbb{H}^{l}(\tilde{X}, \mathscr{K}_{\tilde{X}}^{\bullet})  \ar[r]^{}
& H^{l}(\tilde{X},\mathcal{K}_{\tilde{U}})  \ar[r]^{}
& \mathbb{H}^{l}(\tilde{X},\mathscr{D}_{\tilde{X}}^{\bullet})
\ar[r]^{} &\cdots}
\end{equation*}
Akin to \eqref{comm-long-seq-0}, the blow-up diagram \eqref{b-l-d} induces a commutative diagram:
\begin{equation}\label{rel-long-z-e}
\xymatrix@C=0.3cm{
\cdots\ar[r]^{}
& H^{l-1}(X,\mathcal{K}_{U}) \ar[r]^{}\ar[d]^{\pi^{\ast}_{\tilde{U}}}
& \mathbb{H}^{l-1}(X,\mathscr{D}_{X}^{\bullet})
\ar[r]^{}\ar[d]^{\pi^{\ast}_{\mathscr{D}}}
& \mathbb{H}^{l}(X, \mathscr{K}_{X}^{\bullet})
\ar[r]^{}\ar[d]^{\pi^{\ast}}
& H^{l}(X,\mathcal{K}_{U})  \ar[r]^{}\ar[d]^{\pi^{\ast}_{\tilde{U}}}
& \mathbb{H}^{l}(X,\mathscr{D}_{X}^{\bullet})
\ar[r]^{}\ar[d]^{\pi^{\ast}_{\mathscr{D}}} &\cdots\\
\cdots\ar[r]^{}
& H^{l-1}(\tilde{X},\mathcal{K}_{\tilde{U}}) \ar[r]^{}
& \mathbb{H}^{l-1}(\tilde{X},\mathscr{D}_{\tilde{X}}^{\bullet}) \ar[r]^{}
& \mathbb{H}^{l}(\tilde{X}, \mathscr{K}_{\tilde{X}}^{\bullet})  \ar[r]^{}
& H^{l}(\tilde{X},\mathcal{K}_{\tilde{U}})  \ar[r]^{}
& \mathbb{H}^{l}(\tilde{X},\mathscr{D}_{\tilde{X}}^{\bullet})
\ar[r]^{} &\cdots.}
\end{equation}

On one hand, we have
$$
H^{l}(X,\mathcal{K}_{U})\cong H^{l}_{dR,c}(U)\,\,\,\mathrm{and}\,\,\, H^{l}(\tilde{X},\mathcal{K}_{\tilde{U}})\cong
H^{l}_{dR,c}(\tilde{U}),\,\,\,\mathrm{for\,\,\,any}\,\,\,
l\in\mathbb{N}.
$$
On the other hand, $\pi_{\tilde{U}}:\tilde{U}\rightarrow U$ is biholomorphic.
This follows that $\pi^{\ast}_{\tilde{U}}$ in \eqref{rel-long-z-e} is isomorphic.

Now we claim that $\pi^{\ast}_{\mathscr{D}}$ is also isomorphic.
Note that $\mathbb{H}(X,\mathscr{D}^{\bullet}_{X})$
and $\mathbb{H}(\tilde{X},\mathscr{D}^{\bullet}_{\tilde{X}})$
are isomorphic to the total cohomologies of the bounded double complexes
$(\mathbb{K}_{X}^{\bullet,\bullet}; D_{1}, D_{2})$
and
$(\mathbb{K}_{\tilde{X}}^{\bullet,\bullet}; D_{1}, D_{2})$,
respectively (see Appendix \ref{dou-com}).
According to \cite[Theorem 14.14]{BT82}, there exists a spectral sequence
$\{E_{r},d_{r}\}$ converging to the total cohomology
$H(\mathbb{K}^{\bullet}_{X})$ such that
\begin{equation*}
E^{s,t}_{1}=H^{s,t}_{D_{2}}(\mathbb{K}^{\bullet}_{X})
=
H^{t}(\mathbb{K}_{X}^{s,\bullet})
\cong
H^{t}(X,\mathscr{D}^{s}_{X}).
\end{equation*}
Likewise, the double complex $(\mathbb{K}_{\tilde{X}}^{\bullet,\bullet};\tilde{D}_{1}, \tilde{D}_{2})$ admits a
spectral sequence, denoted by
$\{\tilde{E}_{r},\tilde{d}_{r}\}$,
with the first term
\begin{equation*}
\tilde{E}^{s,t}_{1}
=H^{s,t}_{\tilde{D}_{2}}(\mathbb{K}^{\bullet}_{\tilde{X}})
=
H^{t}(\mathbb{K}_{\tilde{X}}^{s,\bullet})
\cong
H^{t}(\tilde{X},\mathscr{D}^{s}_{\tilde{X}}),
\end{equation*}
which is converging to the total cohomology
$H(\mathbb{K}^{\bullet}_{\tilde{X}})$.

Consider the Godement resolution $\mathscr{D}_{\tilde{X}}^{s}\rightarrow \mathcal{G}^{s, \bullet}$.
Then the $r$-th higher direct image is $R^{r}\pi_{\ast}\mathscr{D}_{\tilde{X}}^{s}=
\mathscr{H}^{r}(\pi_{\ast}\mathcal{G}^{s, \bullet})$.
By Lemma \ref{higher-di-img}, we have $R^{r}\pi_{\ast}\mathscr{D}_{\tilde{X}}^{s}=0$
for $r\geq1$, i.e., $\mathscr{H}^{r}(\pi_{\ast}\mathcal{G}^{s, \bullet})=0$.
Equivalently, this means that
$\pi_{\ast}\mathcal{G}^{s, \bullet}$
is an \emph{exact} sheaf complex and hence a flasque resolution of $\pi_{\ast}\mathscr{D}_{\tilde{X}}^{s}$.
As a result, by definition,
we derive the canonical isomorphisms
\begin{equation}\label{more-iso}
H^{l}(X, \pi_{\ast}\mathscr{D}_{\tilde{X}}^{s})
=
H^{l}(\Gamma(X, \pi_{\ast}\mathcal{G}^{s, \bullet}))
=
H^{l}(\Gamma(\tilde{X}, \mathcal{G}^{s, \bullet}))
=
H^{l}(\tilde{X}, \mathscr{D}_{\tilde{X}}^{s}).
\end{equation}
From Lemma \ref{higher-di-img}, we have the isomorphism
$\pi^{\ast}: \mathscr{D}_{X}^{s}\stackrel{\simeq}\longrightarrow \pi_{\ast}\mathscr{D}_{\tilde{X}}^{s}$,
which induces an isomorphism of sheaf cohomologies
\begin{equation}\label{sh-co-iso}
\pi^{\ast}: H^{l}(X, \mathscr{D}_{X}^{s}) \stackrel{\simeq}\longrightarrow H^{l}(X, \pi_{\ast}\mathscr{D}_{\tilde{X}}^{s}).
\end{equation}

Observe that the morphism of double complexes
$
\pi^{\ast}:\mathbb{K}^{\bullet,\bullet}_{X}
\rightarrow
\mathbb{K}^{\bullet,\bullet}_{\tilde{X}}
$
induces a morphism of the total cohomology groups
$\pi^{\ast}:H(\mathbb{K}_{X}^{\bullet})\rightarrow H(\mathbb{K}_{\tilde{X}}^{\bullet})$
and a morphism of spectral sequences
$
\pi^{\ast}_{r}:E_{r}\rightarrow\tilde{E}_{r}
$
for any $r\geq1$.
From \eqref{more-iso} and \eqref{sh-co-iso}, we get an isomorphism
$$
\pi^{\ast}:
H^{t}(X, \mathscr{D}_{X}^{s})
\stackrel{\simeq}\longrightarrow
H^{t}(\tilde{X}, \mathscr{D}_{\tilde{X}}^{s})
$$
for any $0\leq t\leq n$ and $0\leq s\leq q-1$.
This implies that
$\pi^{\ast}_{1}:E_{1}\rightarrow\tilde{E}_{1}$
is isomorphic.
Due to \cite[Theorem 6.4.2]{GS99},
we get $\pi^{\ast}_{r}$ is isomorphic for any $r>1$;
furthermore, the induced morphism of the total cohomologies for the simple complexes is isomorphic, i.e.,
$\pi^{\ast}_{\mathscr{D}}$ is an isomorphism.

As the morphisms $\pi^{\ast}_{\tilde{U}}$ and $\pi^{\ast}_{\mathscr{D}}$ in \eqref{rel-long-z-e} are isomorphic, from the standard Five Lemma, so is the middle one $\pi^{\ast}$.
We have thus proved the lemma.
\end{proof}


\subsection{Case of $p=0$ or $q=0$}
Due to the symmetric property we have $H^{0,q}_{BC}(-)\cong H^{q,0}_{BC}(-)$ and therefore in the following we shall prove
$H^{p,0}_{BC}(\tilde{X})\cong H^{p,0}_{BC}(X)$.
The result can be proved by the same way as used in  \cite[Proposition 1.2]{Uen73} or \cite[Proposition $4.1$ of Chapter $1$]{Popa}.

\begin{proof}[Proof of Theorem \ref{Thm1} $\mathrm{(ii)}$]
From definition, we have
$$
H^{p,0}_{BC}(\tilde{X})=\{\tilde{\alpha}\in\Gamma(\tilde{X},\Omega^{p}_{\tilde{X}})\,|\,\partial\tilde{\alpha}=0\}
$$
and
$$
H^{p,0}_{BC}(X)=\{\alpha\in\Gamma(X,\Omega^{p}_{X})\,|\,\partial\alpha=0\}.
$$
Put $U=X-Z$ and $\tilde{U}=\pi^{-1}(U)$.
Let $j:U\hookrightarrow X$ and $\tilde{j}:\tilde{U}\hookrightarrow\tilde{X}$ be the associated inclusions.
Note that $\pi$ is a proper holomorphic map.
It follows that the morphism
\begin{equation}\label{eq5.1}
\pi^{\ast}:H^{p,0}_{BC}(X)\rightarrow H^{p,0}_{BC}(\tilde{X})
\end{equation}
is injective and hence $h^{p,0}_{BC}(\tilde{X})\geq h^{p,0}_{BC}(X)$.
The next thing to do in the proof is to show that (\ref{eq5.1}) is surjective, namely, $h^{p,0}_{BC}(\tilde{X})\leq h^{p,0}_{BC}(X)$.

Since $\pi|_{\tilde{U}}:\tilde{U}\rightarrow U$ is biholomorphic we get an isomorphism
\begin{equation}\label{equ5.2}
H^{p,0}_{BC}(\tilde{U})\cong H^{p,0}_{BC}(U).
\end{equation}
The inclusion $\tilde{j}$ induces a restriction morphism
\begin{equation}\label{equ5.3}
\tilde{j}^{\ast}:H^{p,0}_{BC}(\tilde{X})\rightarrow H^{p,0}_{BC}(\tilde{U}).
\end{equation}
We claim that (\ref{equ5.3}) is injective.
If the assertion was not true, then there exists a nonzero $\partial$-closed holomorphic $p$-form on $\tilde{X}$ would vanish on the nonempty open subset $\tilde{U}=\tilde{X}-E$, and this leads to a contradiction.
Using the same argument above, we can verify the injectivity of
$$
j^{\ast}:H^{p,0}_{BC}(X)\rightarrow H^{p,0}_{BC}(U).
$$
Let $\alpha\in H^{p,0}_{BC}(U)$, which is a $\partial$-closed holomorphic $p$-form on $U$.
Note that $\textmd{codim}_{\mathbb{C}}Z\geq2$, by the Hartogs Extension Theorem $\alpha$ extends over $Z$ and we get a holomorphic $p$-form $\beta$ on $X$ such that $\beta|_{U}=\alpha$.
Let $\zeta=\partial\beta$, then $\zeta\in\Gamma(X,\Omega^{p+1}_{X})$.
Observe that the restriction map
$$
j^{\ast}:H^{p+1,0}_{BC}(X)\rightarrow H^{p+1,0}_{BC}(U)
$$
is injective and
$\zeta|_{U}=\partial\alpha=0$ we get $\zeta=0$.
This means $\beta\in H^{p,0}_{BC}(X)$ and therefore
$
j^{\ast}:H^{p,0}_{BC}(X)\rightarrow H^{p,0}_{BC}(U)
$
is surjective;
moreover, we get
\begin{equation}\label{equ5.4}
H^{p,0}_{BC}(X)\cong H^{p,0}_{BC}(U).
\end{equation}
Combining (\ref{equ5.2}), \eqref{equ5.3} and (\ref{equ5.4}) we obtain
$
h^{p,0}_{BC}(\tilde{X})\leq h^{p,0}_{BC}(X).
$
\end{proof}
Due to the Weak Factorization Theorem (see Section \ref{invar-ddbar}, Theorem \ref{WFT}),
if an invariant of compact complex manifolds is stable under the blow-ups then it is a bimeromorphic invariant.
As a direct result of Theorem \ref{Thm1} (ii), we get the following result.
\begin{cor}\label{p-0-inv}
If $X$ and $Y$ are two bimeromorphically equivalent compact complex manifolds with the complex dimension $n$,
then for any $0\leq p\leq n$ we have
\begin{equation*}
h_{BC}^{p, 0}(X)=h_{BC}^{p, 0}(Y).
\end{equation*}
\end{cor}
\subsection{Pointed blow-up}

If $Z=\{\textmd{pt}\}$ is a single-pointed space, then we have
\begin{cor}\label{cor3.4}
\begin{equation*}
H_{BC}^{p,q}(\tilde{X}) = \left\lbrace
           \begin{array}{c l}
             H_{BC}^{p,q}(X)\oplus\mathbb{C}, 
             & \text{$1\leq p=q\leq n-1$};\\
             H_{BC}^{p,q}(X),& \mathrm{otherwise}.
           \end{array}
         \right.
\end{equation*}
\end{cor}
\begin{proof}
As $Z=\{\textmd{pt}\}$ the exceptional divisor $E$ is biholomorphic to $\mathbb{C}\mathrm{P}^{n-1}$.
Note that
$H_{BC}^{p,q}(E)\cong H^{p,q}_{\bar{\partial}}(\mathbb{C}\mathrm{P}^{n-1})$,
we have $H^{0,0}_{BC}(\textmd{pt})=\mathbb{C}$ and  $H_{BC}^{p,q}(\textmd{pt})=0$ for other cases.
Observe that the de Rham cohomology ring of $\mathbb{C}\mathrm{P}^{n-1}$ is
$$
H^{\ast}_{dR}(\mathbb{C}\mathrm{P}^{n-1};\mathbb{C})
=\mathbb{C}[x]/(x^{n}).
$$
Because of the Hodge decomposition, we get that the cokernel of the morphism
$$
(\pi|_{E})^{\ast}:H_{\bar{\partial}}^{p,q}(\textmd{pt})\rightarrow H_{\bar{\partial}}^{p,q}(E)
$$
is zero when $(p,q)=(0,0)$ and isomorphic to $H^{p,q}_{\bar{\partial}}(\mathbb{C}\mathrm{P}^{n-1})$ when $(p,q)\neq(0,0)$.
This completes the proof.
\end{proof}
As an application of Corollary \ref{cor3.4}, we can prove the following result.
\begin{prop}
Assume that $X$ is a $\partial\bar{\partial}$-manifold of $\mathrm{dim}_{\mathbb{C}}X=n$.
Let $\pi:\tilde{X}\rightarrow X$ be the blow-up of $X$ at a point.
Then $\tilde{X}$ is a $\partial\bar{\partial}$-manifold.
\end{prop}
\begin{proof}
According to Corollary \ref{cor3.4}, we get
\begin{equation}\label{equ3.14}
h_{BC}^{p,q}(\tilde{X}) = \left\lbrace
           \begin{array}{c l}
             h_{BC}^{p,q}(X)+1, & \text{if $1\leq p=q\leq n-1$};\\
             h_{BC}^{p,q}(X),& \text{otherwise}.
           \end{array}
         \right.
\end{equation}
For any $0\leq k\leq2n$, using (\ref{equ3.14}), the $k$-th revised non-K\"{a}hlerness degree of $\tilde{X}$ is
\begin{eqnarray*}
  N^{k}(\tilde{X})
   &=& h^{k}_{BC}(\tilde{X})-h^{2n-k}_{BC}(\tilde{X})\\
   &=&\sum_{p+q=k}\bigl(h_{BC}^{p,q}(\tilde{X})-h^{n-p,n-q}_{BC}(\tilde{X})\bigr) \\
   &=& \sum_{p+q=k}\bigl(h_{BC}^{p,q}(X)-h^{n-p,n-q}_{BC}(X)\bigr) \\
   &=& h^{k}_{BC}(X)-h^{2n-k}_{BC}(X)\\
   &=& N^{k}(X).
\end{eqnarray*}
Because $X$ is a $\partial\bar{\partial}$-manifold we have $N^{k}(X)=0$, and hence $N^{k}(\tilde{X})=0$.
Thus we arrive at the conclusion.
\end{proof}
\begin{rem}
Compare also \cite[Example 2.3]{ASTT17} for another different proof of the stability of the $\partial\bar{\partial}$-Lemma under the pointed blow-ups.
In particular, in \cite{ASTT17} the authors consider the orbifold case \cite[Theorem 3.1]{ASTT17} and construct some new $\partial\bar{\partial}$-manifolds \cite[Example 3.2]{ASTT17}.
\end{rem}
\begin{rem}
By the duality \eqref{equ2.2} the blow-up formula in Theorem \ref{Thm1} is equivalent to
$$
H^{n-p,n-q}_{A}(\tilde{X})\cong
H^{n-p,n-q}_{A}(X)\bigoplus\bigl(H^{n-p,n-q}_{A}(E)/
(\pi|_{E})^{\ast}H^{n-p,n-q}_{A}(Z)\bigr)
$$
and this gives rise to a blow-up formula of Aeppli cohomolgy groups.
In particular, Alessandrini-Bassanelli \cite{AB95} proved the following formula
$$
H^{1,1}_{A}(\tilde{X})=\pi^{\ast}H^{1,1}_{A}(X)\bigoplus \pi^{\ast}\mathcal{H}(Z)\langle[E]\rangle,
$$
where $\mathcal{H}$ is the sheaf of germs of pluriharmonic functions.
\end{rem}
\section{Non-K\"{a}hlerness degrees of threefolds}\label{invar-ddbar}
In this section we study the bimeromorphic invariance of non-K\"{a}hlerness degrees of compact complex threefolds.

\begin{defn}
Let $X$ be a compact complex space with $\mathrm{dim}_{\mathbb{C}}X=n$.
A \emph{modification} of $X$ is a proper holomorphic map $f:\tilde{X}\rightarrow X$ satisfying:
\begin{itemize}
  \item [(i)] $\tilde{X}$ is a compact complex space with complex dimension $n$;
  \item [(ii)] there exists an analytic subset $S\subset X$ of codimension $\geq2$ such that
              $$
              f:\tilde{X}-E\stackrel{\simeq}{\longrightarrow}X-S
              $$
              is a biholomorphism, where $E:=f^{-1}(S)$ is called the \emph{exceptional set} of the modification;
\end{itemize}
and then clearly $f(\tilde{X})=X$.
\end{defn}

It is important to notice that the blow-up of a compact complex manifold with a smooth center is a special example of modifications;
meanwhile, the modifications are bimeromorphic maps in the following sense.

\begin{defn}
A {\it meromorphic map} $f:X \dashrightarrow Y$ of compact complex spaces is a map $f$ of $X$ to the set of subsets of $Y$ satisfying the following conditions:
\begin{enumerate}
\item[(i)] The graph $G_f:=\{(x, y) \in X\times Y \mid y\in f(x)\}$ is an analytic subset of $X\times Y$;
\item[(ii)] The projection $p_1:G_f \rightarrow X$ is a modification.
\end{enumerate}
Moreover, if the projection $p_2:G_f \rightarrow Y$ is also a modification, then $f:X \dashrightarrow Y$ is called a {\it bimeromorphic map}. We say $X$ and $Y$ are {\it bimeromorphically equivalent} if there exists a bimeromorphic map $f:X \dashrightarrow Y$ between them.
\end{defn}

Suppose now that $f:X \dashrightarrow Y$ is a bimeromorphic map of compact complex manifolds.
By Hironaka's singularity resolution theorem \cite{Hir64},
there exists a compact complex manifold $Z$ with a triangle
$$
\xymatrix{
&Z \ar[ld]_{\pi_1} \ar[rd]^{\pi_2}&\\
X&& Y,
}
$$
where $\rho:Z\rightarrow G_{f}$ is a resolution of the graph $G_{f}$,
and $\pi_i:= p_i\circ \rho $ is a modification.
There exist many examples of bimeromorphic maps which are not modifications and blow-ups; however, Abramovich-Karu-Matsuki-W{\l}odarczyk \cite{AKMW02} showed that any bimeromorphic map between compact complex manifolds
is a composition of finite sequences of blow-ups and blow-downs of compact complex manifolds with smooth centers (see also \cite{Wlo03}).

\begin{thm}[{Weak Factorization Theorem, \cite[Theorem 0.3.1]{AKMW02}}]\label{WFT}
Let $f: \tilde{X} \dashrightarrow X$ be a bimeromorphic map of compact complex manifolds.
Then there exists a diagram
\begin{equation*}
\tilde{X}=Y_0 \mathop{\dashrightarrow}^{\phi_1}  Y_1   \mathop{\dashrightarrow}^{\phi_2} \cdots  \mathop{\dashrightarrow}^{\phi_{l-1}} Y_{l-1}   \mathop{\dashrightarrow}^{\phi_{l}} Y_l=X,
\end{equation*}
where $Y_i$ is a compact complex manifold, $f= \phi_{l} \circ \phi_{l-1} \circ  \cdots \circ \phi_{1}$, and either $\phi_{i}: Y_{i-1}\dashrightarrow Y_i$ or $\phi_{i}^{-1}: Y_{i}\dashrightarrow Y_{i-1}$ is a morphism obtained by blowing up a smooth center.
\end{thm}

Due to the Theorem \ref{AT13} and Theroem \ref{WFT},
to verify the bimeromorphic invariance of the $\partial\bar{\partial}$-Lemma for threefolds,
it suffices to prove an explicit blow-up formula of Bott-Chern cohomology as Corollary \ref{cor3.4}.

\begin{lem}\label{lem4.2}
Let $X$ be a compact complex threefold and $C\subset X$ a smooth curve.
Then for any $0\leq p, q\leq3$ we have
$$
H_{BC}^{p,q}(\tilde{X}) \cong  H_{BC}^{p,q}(X)\bigoplus H^{p-1,q-1}_{\bar{\partial}}(C),
$$
where $\pi:\tilde{X}\rightarrow X$ is the blow-up of $X$ along $C$.
\end{lem}

\begin{proof}
According to Theorem \ref{Thm1},
the assertion holds if and only if the co-kernel of the morphism
$$
(\pi|_{E})^{\ast}:H_{BC}^{p,q}(C)\rightarrow H_{BC}^{p,q}(E)
$$
equals to
$
H^{p-1,q-1}_{\bar{\partial}}(C),
$
for any $0\leq p, q \leq3$.
By definition, we get that the exceptional divisor $E=\pi^{-1}(C)$ is biholomorphic to the projectivization of the normal bundle of $C$, i.e.,
$E=\mathbb{P}(N_{C/X})$.
Since every compact complex curve is K\"{a}hler the $\partial\bar{\partial}$-Lemma holds on $C$.
It follows that the Bott-Chern cohomology of $C$ is isomorphic to its Dolbeault cohomology.
In addition, note that the exceptional divisor $E=\mathbb{P}(N_{C/X})$ is a K\"{a}hler surface in $\tilde{X}$.
This implies
$
H_{BC}^{p,q}(E)\cong H^{p,q}_{\bar{\partial}}(E).
$
Let $S$ be the universal subbundle over $E$ and let $t=c_{1}(S^{\ast})$ be the first Chern class of the dual bundle $S^{\ast}$,
which is a real de Rham cohomology class of $(1,1)$-type.
According to the Leray-Hirsch theorem \cite[Theorem 5.11]{BT82}, the de Rham cohomology $H^{\ast}_{dR}(E;\mathbb{C})$ is a free module over $H^{\ast}_{dR}(C;\mathbb{C})$ with the basis $\{1,t\}$, i.e.,
$$
H^{\ast}_{dR}(E;\mathbb{C})\cong H^{\ast}_{dR}(C;\mathbb{C})\otimes\{1,t\}.
$$
More precisely, we have
\begin{equation*}
H^{k}_{dR}(E;\mathbb{C})=\pi^{\ast}H^{k}_{dR}(C;\mathbb{C})\oplus t\wedge\pi^{\ast}H^{k-2}_{dR}(C;\mathbb{C}),
\end{equation*}
where $k=p+q$.
It follows that the co-kernel of the morphism
$$
(\pi|_{E})^{\ast}:H^{k}_{dR}(C;\mathbb{C})\rightarrow H^{k}_{dR}(E;\mathbb{C})
$$
is isomorphic to
$
H^{k-2}_{dR}(C;\mathbb{C}).
$
Due to the Hodge decomposition and via a careful degree checking we get
$$
H_{\bar{\partial}}^{p,q}(E)/(\pi|_{E})^{\ast}H_{\bar{\partial}}^{p,q}(C)\cong H^{p-1,q-1}_{\bar{\partial}}(C).
$$
and hence
$$
H_{BC}^{p,q}(E)/(\pi|_{E})^{\ast}H_{BC}^{p,q}(C)\cong H^{p-1,q-1}_{\bar{\partial}}(C).
$$
which completes the proof.
\end{proof}

\begin{rem}
In general, if the submanifold $Z$ is in the \emph{class $\mathcal{C}$ of Fujiki},
then the exceptional divisor $E$ is also in the class $\mathcal{C}$ of Fujiki.
Consequently, the $\partial\bar{\partial}$-Lemma holds on $E$ and hence the Bott-Chern cohomology $H^{\bullet,\bullet}_{BC}(E)$ is isomorphic to the Dolbeaut cohomology $H^{\bullet,\bullet}_{\bar{\partial}}(E)$.
According to the Hodge decomposition and the projective bundle formula for de Rham cohomology, following the steps in Lemma \ref{lem4.2}, we can prove the following formula
$$
H_{BC}^{p,q}(\tilde{X}) \cong  H_{BC}^{p,q}(X)\oplus \Big(\bigoplus_{i=1}^{r-1} H^{p-i,q-i}_{\bar{\partial}}(Z) \Big).
$$
\end{rem}

We are ready to prove Theorem \ref{Thm2}.

\begin{proof}[Proof of Theorem \ref{Thm2}]
Let $f:\tilde{X}\dashrightarrow X$ be a bimeromorphic map of compact complex threefolds.
From Theorem \ref{WFT},
there exists a diagram
\begin{equation}\label{WFT-sequence}
\tilde{X}=Y_0 \mathop{\dashrightarrow}^{\phi_1}  Y_1   \mathop{\dashrightarrow}^{\phi_2} \cdots  \mathop{\dashrightarrow}^{\phi_{l-1}} Y_{l-1}   \mathop{\dashrightarrow}^{\phi_{l}} Y_l=X,
\end{equation}
where either $\phi_{i}: Y_{i-1}\dashrightarrow Y_i$ or $\phi_{i}^{-1}: Y_{i}\dashrightarrow Y_{i-1}$ is a morphism obtained by blowing up of a compact complex threefold along a point or a smooth curve.
According to the diagram (\ref{WFT-sequence}), to verify $\Delta^{k}(\tilde{X})=\Delta^{k}(X)$ for any $k\in\{0,1,2,3,4,5,6\}$, it suffices to prove $\Delta^{k}(Y_i)=\Delta^{k}(Y_{i+1})$.

Note that $X$ is a compact complex threefold.
Without loss of the generality,
we may assume that $g: Y\to X$ is a blow-up of $X$ at a point or along a smooth curve
then we divide the proof in two cases.

\paragraph{\textbf{Case 1}}
Assume that $g:Y\rightarrow X$ is a blow-up of $X$ at a point.
From Corollary \ref{cor3.4} we have
\begin{equation}\label{b-l-pt}
h_{BC}^{p,q}(Y) = \left\lbrace
           \begin{array}{c l}
             h_{BC}^{p,q}(X)+1, & \text{if $1\leq p,q\leq2$ and $p=q$};\\
             h_{BC}^{p,q}(X), &\text{otherwise}.
           \end{array}
         \right.
\end{equation}
If $k\in\{0,1,3,5,6\}$, then from (\ref{b-l-pt}) and the de Rham blow-up formula \cite[Theorem 7.31]{Voi02} we obtain
\begin{eqnarray*}
  \Delta^{k}(Y)&=&h^{k}_{BC}(Y)-h^{6-k}_{BC}(Y)-2b_{k}(Y)\\
   &=& h^{k}_{BC}(X)-h^{6-k}_{BC}(X)-2b_{k}(X)\\
   &=& \Delta^{k}(X).
\end{eqnarray*}
Moreover, in the case of $k\in\{2,4\}$, via a direct check we have
\begin{eqnarray*}
                    \Delta^{k}(Y) &=& h^{k}_{BC}(Y)+h^{6-k}_{BC}(Y)-2b_{k}(Y) \\
                     &=& (h^{k}_{BC}(X)+1)+(h^{6-k}_{BC}(X)+1)-2(b_{k}(X)+1) \\
                     &=& \Delta^{k}(X).
\end{eqnarray*}

\paragraph{\textbf{Case 2}}
Now we assume that $g$ is the blow-up of $X$ along a smooth curve $C$.
On one hand, since $C$ is a smooth curve,
by Lemma \ref{lem4.2} and for the dimension reason we get
\begin{equation}\label{equ4.3}
h_{BC}^{p,q}(Y)=h_{BC}^{p,q}(X)+h^{p-1,q-1}_{\bar{\partial}}(C)
\end{equation}
and
\begin{equation}\label{equ4.4}
h^{3-p,3-q}_{BC}(Y)=h^{3-p,3-q}_{BC}(X)+h^{2-p,2-q}_{\bar{\partial}}(C)
\end{equation}
where $2\leq p+q\leq4$;
otherwise,
$
h_{BC}^{p,q}(Y)=h_{BC}^{p,q}(X).
$
On the other hand, according to the de Rham blow-up formula we get
\begin{equation}\label{betti}
b_{k}(Y) = \left\lbrace
           \begin{array}{c l}
             b_{k}(X)+b_{k-2}(C), & \text{if $2\leq k\leq4$};\\
             b_{k}(X), & \text{otherwise}.
           \end{array}
         \right.
\end{equation}
If $k\in\{0,1,3,5,6\}$ then via a straightforward computation we can verify that the equality $\Delta^{k}(Y)=\Delta^{k}(X)$ holds.
If $k=2$ or $k=4$, then from the Poincar\'{e} duality on $C$ we have $b_{4-k}(C)-b_{k-2}(C)=0$;
furthermore, by (\ref{equ4.3})-(\ref{betti}) we get
\begin{eqnarray*}
                    \Delta^{k}(Y) &=& h^{k}_{BC}(Y)+h^{6-k}_{BC}(Y)-2b_{k}(Y) \\
                     &=& \bigl(h^{k}_{BC}(X)+b_{k-2}(C)\bigr)+\bigl(h^{6-k}_{BC}(X)+b_{4-k}(C)\bigr)-2\bigl(b_{k}(X)+b_{k-2}(C)\bigr) \\
                     &=& \Delta^{k}(X)+\bigl(b_{4-k}(C)-b_{k-2}(C)\bigr)\\
                     &=& \Delta^{k}(X).
\end{eqnarray*}
The results in \textbf{Case 1} and \textbf{Case 2} mean $\Delta^{k}(Y)=\Delta^{k}(X)$ for any $k\in\{0,1,2,3,4,5,6\}$.
Likewise, we can show that $\Delta^{k}(Y_{i})=\Delta^{k}(Y_{i+1})$.
This implies $\Delta^{k}(\tilde{X})=\Delta^{k}(X)$, i.e., the non-K\"{a}hlerness degrees for compact complex threefolds are bimeromorphic invariants.
According to Theorem \ref{AT13}, we get that $\tilde{X}$ is a
$\partial\bar{\partial}$-manifold if and only if $X$ is a $\partial\bar{\partial}$-manifold and this completes the proof.
\end{proof}

For compact complex threefolds, Theorem \ref{Thm2} can be thought of as a generalization of \cite[Theorem 5.22]{DGMS75}.
In particular, following the steps in the proof of Theorem \ref{Thm2} we can show that the revised non-K\"{a}hlerness degrees are also bimeromorphic invariants for threefolds.
\begin{ex}[Iwasawa manifold]\label{ex4.2}
Recall that the 3-dimensional complex Heisenberg group $\mathbb{H}(3;\mathbb{C})$ is defined to be the Lie subgroup of $\mathrm{GL}(3;\mathbb{C})$ whose elements have the form
$$
g=\left(
  \begin{array}{ccc}
    1 & a & b \\
    0 & 1 & c \\
    0 & 0 & 1 \\
  \end{array}
\right)
$$
where $a,b,c$ are complex numbers.
It is noteworthy that $\mathbb{H}(3;\mathbb{C})$ is a connected and simply connected complex nilpotent Lie group.
Let $\Gamma\subset\mathbb{H}(3;\mathbb{C})$ be a discrete subgroup such that all entries are Gaussian integers.
The Iwasawa manifold is defined to be quotient space
$
\mathbb{I}_{3}:=\mathbb{H}(3;\mathbb{C})/\Gamma,
$
which is a compact 6-dimensional smooth manifold.
Besides, there exists a $\mathbb{H}(3;\mathbb{C})$-invariant complex structure $J_{0}$ on $\mathbb{I}_{3}$ determined by the standard complex structure on $\mathbb{C}^{3}$.
This enables $(\mathbb{I}_{3},J_{0})$ to be a \emph{non-K\"{a}hler, no-formal, and holomorphically parallelizable} manifold.
In 1975, Nakamura \cite{Nak75} proved that the holomorphically parallelizable property of $\mathbb{I}_{3}$ is not stable under small deformations and the deformed objects of $\mathbb{I}_{3}$ can be divided into three classes.
Via computing the Bott-Chern cohomology of Iwasawa manifold $\mathbb{I}_{3}$ and of its small deformations, Angella \cite{Ang13} refined Nakamura's classification.
In particular, we have the following table (cf. \cite{AT13,AT17}).

\begin{table}[h]
\centering
\label{my-label}
\small{
\begin{tabular}{|l|l|l|l|l|l|l|}
\hline
& $\mathbb{I}_{3}$ & (i) & (ii.a) & (ii.b) & (iii.a) & (iii.b) \\ \hline
$\Delta^{1}=\Delta^{5}$ &  2  &   2  &    2    &    2    &    2     &    2     \\ \hline
$\Delta^{2}=\Delta^{4}$ &  6  &   6  &    3    &    2    &    1     &    0     \\ \hline
$\Delta^{3}$ &  8  &   8  &    8    &    8    &    8     &    8     \\ \hline
\end{tabular}
}
\centering

\;
\;
\;

\caption{Non-K\"{a}hlerness degrees of $\mathbb{I}_{3}$ and its small deformations}
\end{table}
\end{ex}
\section{Concluding remark}\label{remark}

Based on the Weak Factorization Theorem (Theorem \ref{WFT}), to verify the bimeromorphic invariance of the $\partial\bar{\partial}$-Lemma in higher dimension case, we need to deal with two problems:
to confirm whether the $\partial\bar{\partial}$-Lemma survive the closed complex submanifolds
and prove a general \emph{projective bundle formula} for Bott-Chern cohomology.
More precisely, let $X$ be a compact complex manifold with $\mathrm{dim}_{\mathbb{C}}\,X=n>3$.
It is unknown whether the following explicit formula still holds
\begin{equation}\label{blow-up-formula}
H_{BC}^{p,q}(\tilde{X}) \cong  H_{BC}^{p,q}(X)\oplus \Big(\bigoplus_{i=1}^{r-1} H^{p-i,q-i}_{BC}(Z) \Big);
\end{equation}
however, it seems to be reasonable.
If the formula (\ref{blow-up-formula}) is true
and any closed complex submanifold of a $\partial\bar{\partial}$-manifold
is also a $\partial\bar{\partial}$-manifold,
then by Theorem \ref{thm2.4} we can verify that the $\partial\bar{\partial}$-Lemma is a bimeromorphic invariant of compact complex manifolds.

For a compact non-K\"{a}hler complex surface the non-K\"{a}hlerness degrees are trivial bimeromorphic invariants since $\Delta^{1}=0$ and $\Delta^{2}=2$.
Let $X$ be a compact complex threefold.
By Theorem \ref{Thm2} we know that for each $k\in\{0,1,2,3,4,5,6\}$, the non-K\"{a}hlerness degree $\Delta^{k}(X)$ is a bimeromorphic invariant.
Example \ref{ex4.2} means that for compact non-K\"{a}hler threefolds the non-K\"{a}hlerness degrees, as bimeromorphic invariants, are \emph{non-trivial}.
As a direct consequence, we get that the Iwasawa manifold $\mathbb{I}_{3}$ and its small deformations of class (ii) (resp. class (iii)) are not bimeromorphically equivalent.
From the bimeromorphic geometry point of view, a natural problem is:
\begin{prob}
Whether there exists a uniform upper bound for (revised) non-K\"{a}hlerness degrees of compact non-K\"{a}hler threefolds;
moreover, how to classify compact non-K\"{a}hler threefolds (up to the bimeromorphic equivalence) with respect to the (revised) non-K\"{a}hlerness degrees?
\end{prob}
Note that our proof of the bimeromorphic invariance of non-K\"{a}hlerness degrees for threefolds depends on the explicit Bott-Chern blow-up formula at a point or along a smooth curve.
Unfortunately, our proof does not enable us to show the bimeromorphic invariance of non-K\"{a}hlerness degrees for compact complex manifolds with complex dimensions larger than 3.


\subsection*{Update (August 2018)}
The Bott-Chern projective bundle formula has been claimed by Dr. J. Stelzig in \cite{Ste18} using a structure theory of double complexes \cite{Ste}.
Combining Theorem \ref{Thm1} with the result of Stelzig,
one can show that the formula \eqref{blow-up-formula} holds.
Consequently, Problem \ref{problem} descends to the following fundamental problem which is still open; see also
\cite{Ale17, ASTT17, Fri17, RYY17}.

\begin{prob}
Let $X$ be a $\partial\bar{\partial}$-manifold,
is it true that the $\partial\bar{\partial}$-Lemma holds on any closed complex submanifold $Z$ of $X$?
\end{prob}
Notice that if the above problem is confirmed positively then from the blow-up formula \eqref{blow-up-formula}, Theorem \ref{AT13} (or \ref{thm2.4}) and Weak Factorization Theorem we can verify that the $\partial\bar{\partial}$-Lemma is a bimeromorphic invariant of compact complex manifolds.
\appendix
\section{Properties of $\mathscr{D}^{\bullet}_{\tilde{X}}$}
\subsection{Double complex reinterpretation of $\mathscr{D}^{\bullet}_{\tilde{X}}$}
\label{dou-com}
With the situation in Section \ref{blow-up-formula-BC},
we consider the sheaf complex
\begin{equation*}
\xymatrix@C=0.4cm{\mathscr{D}^{\bullet}_{X}:\,
0\ar[r]^{} & \mathcal{K}_{X}^{0}\oplus \bar{\mathcal{K}}_{X}^{0} \ar[r]^{} &\mathcal{K}_{X}^{1}\oplus \bar{\mathcal{K}}_{X}^{1}
  \ar[r]^{}& \cdots  \ar[r]^{}& \mathcal{K}_{X}^{p-1}\oplus \bar{\mathcal{K}}_{X}^{p-1} \ar[r]^{} & \bar{\mathcal{K}}_{X}^{p} \ar[r]^{} & \cdots  \bar{\mathcal{K}}_{X}^{q-1} \ar[r]^{} & 0.}
\end{equation*}
Set
$
\K_{X}^{s,t}=\ker\big(\E_{X}^{s,t} \stackrel{\varphi}\longrightarrow  \iota_{\ast}\E_{Z}^{s,t}\big).
$
Due to the Dolbeault-Grothendieck lemma \cite[Lemma 3.29]{Dem12}, we have a \emph{fine} resolution
\begin{equation}\label{re-bo-ch-reso}
\xymatrix@C=0.4cm{
\mathcal{K}_{X}^{0,\bullet}\oplus \bar{\mathcal{K}}_{X}^{0,\bullet} \ar[r]^{} &\mathcal{K}_{X}^{1,\bullet}\oplus \bar{\mathcal{K}}_{X}^{1,\bullet}
\ar[r]^{}& \cdots  \ar[r]^{}& \mathcal{K}_{X}^{p-1,\bullet}\oplus
\bar{\mathcal{K}}_{X}^{p-1,\bullet} \ar[r]^{} & \bar{\mathcal{K}}_{X}^{p,\bullet} \ar[r]^{} & \cdots \ar[r]^{} & \bar{\mathcal{K}}_{X}^{q-1,\bullet} & \\
 \mathcal{K}^{0}_{X}\oplus
\bar{\mathcal{K}}^{0}_{X}
\ar@{^{(}->}[u]\ar[r]^{}
&\mathcal{K}_{X}^{1}\oplus \bar{\mathcal{K}}_{X}^{1}
\ar@{^{(}->}[u]\ar[r]^{}
& \cdots  \ar[r]^{}& \mathcal{K}_{X}^{p-1}\oplus \bar{\mathcal{K}}_{X}^{p-1}
\ar@{^{(}->}[u]\ar[r]^{}
& \bar{\mathcal{K}}_{X}^{p}
\ar@{^{(}->}[u]\ar[r]^{} & \cdots \ar[r]^{} & \bar{\mathcal{K}}_{X}^{q-1}
\ar@{^{(}->}[u].}
\end{equation}

For the simplicity, we set
$\mathrm{K}^{s,t}=\Gamma(X, \K_{X}^{s,t})$,
and $\bar{\mathrm{K}}^{s,t}=\Gamma(X, \bar{\K}_{X}^{s,t})$.
Moreover, for any $0\leq t\leq n$, we define
$$
\mathbb{K}_{X}^{s,t}=\mathrm{K}^{s,t}\oplus\bar{\mathrm{K}}^{s,t},\;\;
\mathrm{for\,\,\,any}\,\,0\leq s\leq p-1,
$$
$$
\mathbb{K}_{X}^{s,t}=\bar{\mathrm{K}}^{s,t},
\;\;\quad\quad\mathrm{for\,\,\,any}\,\,p\leq s\leq q-1.
$$
By taking the global sections of the resolutions
\eqref{re-bo-ch-reso},
we obtain a double complex $\mathbb{K}^{\bullet,\bullet}_{X}$
with the horizontal and vertical differentials
$$
D_{1}=\partial\oplus \bar{\partial}:
\mathbb{K}_{X}^{s,t}\rightarrow  \mathbb{K}_{X}^{s+1,t},\;\;\;\;\;\;
D_{2}=\bar{\partial}\oplus  \partial:
\mathbb{K}_{X}^{s,t}\rightarrow  \mathbb{K}_{X}^{s,t+1},
$$
satisfying $D_{2}^{2}=0=D_{1}^{2}$ and $D_{2}D_{1}+D_{1}D_{2}=0$.
Denote the simple complex by
$$
\mathbb{K}_{X}^{l}=\bigoplus_{s+t=l}\mathbb{K}_{X}^{s,t}.
$$
Then we have an isomorphism
$$
\mathbb{H}(X,\mathscr{D}^{\bullet}_{X})\cong
H(\mathbb{K}^{\bullet}_{X}).
$$
Analogously, we can construct a double complex $\mathbb{K}^{\bullet,\bullet}_{\tilde{X}}$ associated to the pair $(\tilde{X},E)$.
Moreover, there  exists an isomorphism
$$
\mathbb{H}(\tilde{X},\mathscr{D}^{\bullet}_{\tilde{X}})\cong
H(\mathbb{K}^{\bullet}_{\tilde{X}}).
$$
\subsection{Higher direct images}
With the situation in Section \ref{blow-up-formula-BC},
we have the blow-up diagram
\begin{equation*}\label{b-l-d-appendix}
\xymatrix{
E \ar[d]_{\rho:=\pi|_{E}} \ar@{^{(}->}[r]^{\tilde{\iota}} & \tilde{X}\ar[d]^{\pi}\\
 Z \ar@{^{(}->}[r]^{\iota} & X.
}
\end{equation*}
Recall that the pullback of differential forms defines a natural surjective morphism of sheaves
$\tilde{\varphi}:
\Omega_{\tilde{X}}^{s}
\rightarrow
\tilde{\iota}_{\ast}\Omega_{E}^{s}$.
Thus, there exists a short exact sequence
\begin{equation}\label{re-E}
\xymatrix{
  0 \ar[r] & \K^{s}_{\tilde{X}} \ar[r]^{} & \Omega_{\tilde{X}}^{s} \ar[r]^{\tilde{\varphi}} & \tilde{\iota}_{\ast}\Omega_{E}^{s} \ar[r] & 0,}
\end{equation}
and we call $\K^{s}_{\tilde{X}}$ the \emph{relative Dolbeault sheaf} with respect to $(\tilde{X},E)$.
The sequence \eqref{re-E} induces a long exact sequence of higher direct images along $\pi$:
\begin{equation*}
\small{
\xymatrix@C=0.5cm{
\cdots\ar[r] &   R^{r-1}\pi_{\ast}\Omega^{s}_{\tilde{X}}
  \ar[r]^{} & R^{r-1}\pi_{\ast}\tilde{\iota}_{\ast}\Omega^{s}_{E}
  \ar[r] & R^{r}\pi_{\ast}\K^{s}_{\tilde{X}}
  \ar[r]^{} & R^{r}\pi_{\ast}\Omega^{s}_{\tilde{X}}
  \ar[r]^{}& R^{r}\pi_{\ast}\tilde{\iota}_{\ast}\Omega^{s}_{E}
   \ar[r] &\cdots.}
 }
\end{equation*}
In general, the higher direct image
$R^{r}\pi_{\ast}\Omega^{s}_{\tilde{X}}$ is not vanishing,
unless $s=0$ or $n=\mathrm{dim}_{C}\,\tilde{X}$.
However, the higher direct images of the relative Dolbeault sheaves always vanish.

The following proposition was pointed out to us by Dr. L. Meng \cite[Lemma 2.4]{Men18}
and we figured out a detailed proof for the bundle-valued case in \cite[Lemma 4.4]{RYY}.

\begin{prop}\label{hi-dir-0}
$R^{r}\pi_{\ast}\K_{\tilde{X}}^{s}=0$ for any $r\geq 1$.
\end{prop}

In particular, for the sheaf
$\mathscr{D}_{X}^{s}=
\K^{s}_{\tilde{X}}\oplus\bar{\K}^{s}_{\tilde{X}}$
the following result holds.
\begin{lem}\label{higher-di-img}
For any $s\geq 0$,
$\pi^{\ast}: \mathscr{D}_{X}^{s} \stackrel{\simeq}\longrightarrow \pi_{\ast}\mathscr{D}_{\tilde{X}}^{s}$
and $R^{r}\pi_{\ast}\mathscr{D}_{\tilde{X}}^{s}=0$ for $r\geq 1$.
\end{lem}
\begin{proof}
Note that the direct image functor commutes with the direct sum
$$
\pi_{\ast}(\K^{s}_{\tilde{X}}\oplus\bar{\K}^{s}_{\tilde{X}})
=
\pi_{\ast}\K^{s}_{\tilde{X}}\oplus\pi_{\ast}
\bar{\K}^{s}_{\tilde{X}}.
$$
To prove the assertion holds, it suffices to show that
$\pi_{\ast}\K^{s}_{\tilde{X}}\cong\K^{s}_{X}$
and then we get
$\pi_{\ast}\bar{\K}^{s}_{\tilde{X}}\cong \bar{\K}^{s}_{X}$
by complex conjugation.
More precisely, we only need to show that the induced morphism of stalks
$$
\pi^{\ast}_{x}: (\mathscr{D}_{X}^{s})_{x}
\longrightarrow
(\pi_{\ast}\mathscr{D}_{\tilde{X}}^{s})_{x},
\,\,\,\mathrm{for\,\,any}\,\,x\in X
$$
is isomorphic.

By definition, for any open subset $V\subset X$ we have
$$
\Gamma(V,\pi_{\ast}\K^{s}_{\tilde{X}})
=
\{\tilde{\alpha}\in \Gamma(\tilde{V}, \Omega^{s}(\tilde{V})) \,|\,(\tilde{\iota}_{\tilde{V}\cap E})^{\ast}\tilde{\alpha}=0\},\,\,
\textmd{if}\,\,\,V\cap Z\neq\emptyset;
$$
and
$$
\Gamma(V,\pi_{\ast}\K^{s}_{\tilde{X}})
=
\Gamma(\tilde{V}, \K_{\tilde{X}}^{s})\cong\Gamma(V, \K_{X}^{s}),\,\,\,
\textmd{if}\,\,\,V\cap Z=\emptyset;
$$
where $\tilde{V}=\pi^{-1}(V)$.
If $\beta$ is a holomorphic $s$-form on $V$ such that $(\iota_{V\cap Z})^{\ast}\beta=0$,
the the pullback $\tilde{\beta}:=(\pi|_{\tilde{V}})^{\ast}\beta$ is a holomorphic $s$-form on $\tilde{V}$ that is satisfying
$(\tilde{\iota}_{\tilde{V}\cap E})^{\ast}\tilde{\beta}=0.$
This implies that the pullback $\pi^{\ast}$ induces a sheaf morphism
$\pi^{\ast}: \K_{X}^{s}\rightarrow \pi_{\ast}\K_{\tilde{X}}^{s}$
by assigning the map
\begin{equation*}
  \pi^{\ast}(V):\Gamma(V, \K_{X}^{s})
  \rightarrow
  \Gamma(V, \pi_{\ast}\K_{\tilde{X}}^{s}),\,\,\,
  \alpha\mapsto (\pi|_{\tilde{V}})^{\ast} \alpha
\end{equation*}
to each open subset $V\subset X$.
Moreover, from \cite[Lemma 4.4]{RYY} it can be shown that the sheaf morphism
$\pi^{\ast}: \K_{X}^{s}\rightarrow \pi_{\ast}\K_{\tilde{X}}^{s}$
is actually isomorphic.

It remains to prove
$R^{r}\pi_{\ast}\mathscr{D}_{\tilde{X}}^{s}=0$
for any $r\geq 1$.
According to Proposition \ref{hi-dir-0}, we have
$R^{r}\pi_{\ast}\K^{s}_{\tilde{X}}=0$ for any $r\geq 1$.
We claim $R^{r}\pi_{\ast}\bar{\K}^{s}_{\tilde{X}}=0$ for any $r\geq 1$.
From definition, for a given sheaf $\mathcal{F}$ on $\tilde{X}$,
its higher direct image $R^{r}\pi_{\ast}\mathcal{F}$ is
the sheafification of the presheaf
$$
V\mapsto H^{r}(\tilde{V}, \mathcal{F})
$$
for any open set $V\subset X$ and $\tilde{V}:=\pi^{-1}(V)$.
On account of $R^{r}\pi_{\ast}\K_{\tilde{X}}^{s}=0$,
for every point $x\in X$ there is an open neighborhood $V$ of $x$ such that
$
H^{r}(\tilde{V}, \K_{\tilde{V}}^{s})=0.
$
It means that the complex
$(\Gamma(\tilde{V}, \K_{\tilde{V}}^{s, \bullet}), \bar{\partial})$
is exact, and hence its complex conjugation
$(\Gamma(\tilde{V}, \bar{\K}_{\tilde{V}}^{s, \bullet}), \partial)$.
This implies
$H^{r}(\tilde{V}, \bar{\K}_{\tilde{V}}^{s})=0$
which concludes the claim
$R^{r}\pi_{\ast}\bar{\K}_{\tilde{V}}^{s}=0$.
By the fact that the higher direct images commutes with direct sums,
we have
$$
R^{r}\pi_{\ast} \mathscr{D}_{\tilde{X}}^{s}
=R^{r}\pi_{\ast}(\K^{s}_{\tilde{X}}\oplus \bar{\K}^{s}_{\tilde{X}})
=R^{r}\pi_{\ast}\K^{s}_{\tilde{X}} \oplus R^{r}\pi_{\ast}\bar{\K}^{s}_{\tilde{X}}
=0\oplus 0.
$$
This finishes the proof of Lemma \ref{higher-di-img}.
\end{proof}


\end{document}